\documentclass[a4paper,10pt,reqno, english]{amsart}
\usepackage[utf8]{inputenc}
\usepackage[T1]{fontenc}    
\usepackage{geometry} 
\usepackage{amsmath}        
\usepackage{amssymb}
\usepackage{amsthm}
\usepackage{paralist}
\usepackage{hyperref}
\usepackage{graphicx}
\usepackage{color}
\usepackage{blindtext}
\usepackage{mathtools} 
\usepackage{bbm}           
\usepackage{url}            
\usepackage{calc}          
\usepackage[mathscr]{eucal}
\usepackage{ae,aecompl}  
\usepackage{xspace} 
\usepackage{floatpag}      
\usepackage{mathtools} 
\usepackage{caption}
\usepackage{subcaption}
\usepackage{multirow}
\usepackage[font=footnotesize]{subcaption}
\usepackage{float}
\restylefloat{table}
\hypersetup{
  colorlinks = true,
  citecolor=black,
  linkcolor= black,
  urlcolor = black,
  pdftitle = {Characterizing face and flag vector pairs for polytopes}
  pdfauthor = {},
  pdfkeywords = {},
  pdfsubject = {},
  pdfpagemode = UseNone
}
\graphicspath{{./pic/}}
\DeclarePairedDelimiter\floor{\lfloor}{\rfloor}

\numberwithin{equation}{section}
 
\newcommand{\N}{\mathbb{N}}
\newcommand{\Z}{\mathbb{Z}}
\newcommand{\R}{\mathbb{R}}
\newtheorem{theorem}{Theorem}[section]
\newtheorem{lemma}[theorem]{Lemma}
\theoremstyle{definition}
\newtheorem{definition}[theorem]{Definition} 

\newcommand\conv{\operatorname{conv}}
\newcommand\ver{\operatorname{vert}}
\newcommand{\fset}{\mathcal{F}}

\newcommand\PiFset[3]{\Pi_{#2}(\fset^{#1}_{#3})}
                
\begin{document}

\title{Characterizing face and flag vector pairs for polytopes}

\author[Sj\"oberg]{Hannah Sj\"oberg} 
\address{Institute of Mathematics, Freie Universität Berlin, Arnimallee 2, 14195 Berlin, Germany}
\email{sjoberg@math.fu-berlin.de}

\author[Ziegler]{G\"unter M. Ziegler} 
\address{Institute of Mathematics, Freie Universität Berlin, Arnimallee 2, 14195 Berlin, Germany}
\email{ziegler@math.fu-berlin.de}

\date{October 18, 2018}
\thanks{In memory of Branko Grünbaum (1929--2018)}

\begin{abstract}\noindent%
Grünbaum, Barnette, and Reay in 1974 completed the characterization of the
pairs $(f_i,f_j)$ of face numbers of $4$-dimensional polytopes. 

Here we obtain a complete characterization of the pairs of flag numbers $(f_0,f_{03})$
for $4$-polytopes. Furthermore, we describe the pairs of face numbers $(f_0,f_{d-1})$ 
for $d$-polytopes; this description is complete for even $d\ge6$ except for finitely many 
exceptional pairs  that are ``small'' in a well-defined sense, while for odd $d$ 
we show that there are also ``large'' exceptional pairs.

Our proofs rely on the insight that ``small'' pairs need to be defined and to be treated separately;
in the $4$-dimensional case, these may be characterized with the help of the
characterizations of the $4$-polytopes with at most $8$ vertices by Altshuler and Steinberg (1984).
\end{abstract}

\maketitle

\section{Introduction}

For a $d$-dimensional polytope $P$, let $f_i=f_i(P)$ 
denote the number of $i$-dimensional faces of~$P$,
and for $S \subseteq \{0, \dots ,d-1\}$,
let $f_S=f_S(P)$ denote the number of chains $F_1\subset\dots\subset F_r$
of faces of $P$ with $\{ \dim F_1,\dots, \dim F_r \} = S$.
The \emph{$f$-vector}  of $P$ is then $(f_0, f_1, \dots ,f_{d-1})$, and 
the \emph{flag vector} of $P$ is $(f_S)_{S \subseteq \{0, \dots , d-1\}}$.
The set of all $f$-vectors of $d$-polytopes is $\fset^d\subset\Z^d$.
(Due to the Euler equation, this set lies on a hyperplane in $\R^d$; it spans
this hyperplane by Grünbaum~\cite[Sect.~8.1]{Grunbaum}.)

While the $f$-vector set $\fset^3$ of $3$-polytopes was characterized (easily) by Steinitz~\cite{Steinitz3} in~1906,
a complete characterization of $\fset^d$ is out of reach for any $d\ge4$. 
For $d=4$, the projections 
of the $f$-vector set $\fset^4\subset\Z^4$ onto two of the four coordinates
have been determined in 1967--1974 by Grünbaum~\cite[Sect.~10.4]{Grunbaum}, Barnette--Reay~\cite{BarnetteReay}
and Barnette~\cite{Barnette}. We will review these results in Section~\ref{sec:facepairs}.
{(Note that a complete characterization of the larger set of flag $f$-vectors of regular CW 3-spheres not necessarily
satisfying the intersection property---so their face posets need not be lattices, in which case they
are not polytopal---was provided by Murai and Nevo \cite[Cor.~3.5]{MuraiNevo14}.)
}

This paper provides new results about coordinate projections of $f$-vector and flag vector sets:
The first part is an extension to the flag vectors of $4$-polytopes. 
In particular, in Theorem~\ref{main_thm1} we fully
characterize the projection of the set of 
all flag vectors of $4$-polytopes to the two coordinates 
$f_0$ and $f_{03}$.
Our proof makes use of the classification of all combinatorial types of $4$-polytopes with up to 
eight vertices by Altshuler and Steinberg~\cite{AS84,AS85}. We have not used the classification
of the $4$-polytopes with nine vertices recently provided by Firsching~\cite{Firsching}.

In the second part we look at the set $\fset^d$ of $f$-vectors of $d$-dimensional polytopes, for $d\ge5$.
Here even a complete characterization of the projection $\PiFset{d}{0,d-1}{}\subset\Z^2$
to the coordinates $f_0$ and $f_{d-1}$ seems impossible. 
We call $(n,m)$ a \emph{polytopal pair} if $(n,m)\in\PiFset{d}{0,d-1}{}$, that
is, if there is a $d$-polytope $P$ with $f_0(P)=n$ and $f_{d-1}(P)=m$.
These polytopal pairs must satisfy the \emph{UBT inequality} 
$m\le f_{d-1}(C_d(n))$ given by the Upper Bound Theorem~\cite{McMullen}~\cite[Sect.~8.4]{Ziegler},
where $C_d(n)$ denotes a $d$-dimensional cyclic polytope with $n$ vertices, and also 
$n\le f_{d-1}(C_d(m))$, by duality.

Our second main result, Theorem~\ref{thm:f0fd-1_evendim}, states that for even $d\ge4$,
every $(n,m)$ satisfying the UBT inequalities as well as $n+m\ge\binom{3d+1}{\floor{d/2}}$ is a polytopal pair.
However, for even $d\ge6$, there are pairs $(n,m)$ with  $n+m < \binom{3d+1}{\floor{d/2}}$ that satisfy the UBT inequalities, but for which there is no polytope: We call these \emph{small exceptional pairs}.
Theorem~\ref{thm:f0fd-1_odddim} states, in contrast, that for every odd $d\ge5$ there are also
\emph{arbitrarily large exceptional pairs}.

\section{Face and flag vector pairs for $4$-polytopes}\label{sec:4polytopes}
\subsection{Face vector pairs for $4$-polytopes}\label{sec:facepairs}
The $2$-dimensional coordinate projections $\PiFset{4}{i,j}{}$ of the set of $f$-vectors of $4$-polytopes 
to the coordinate planes, as determined by Grünbaum, Barnette and Reay, are given
by the following theorems. See also Figure~\ref{plot:fvec}.

\begin{theorem}[Grünbaum {\cite[Thm.~10.4.1]{Grunbaum}}]\label{thm:f0f3}
The set of $f$-vector pairs $(f_0,f_3)$ of $4$-polytopes is equal to
\begin{align*}
\PiFset{4}{0,3}{}=\{(f_0,f_3)\in \Z^2:
5 &\le f_0 \le \tfrac{1}{2} f_3(f_3 - 3),\\
5 &\le f_3 \le \tfrac{1}{2} f_0(f_0 - 3)\}.
\end{align*}
\end{theorem}

\begin{theorem}[Grünbaum {\cite[Thm.~10.4.2]{Grunbaum}}] \label{thm:f0f1}
The set of $f$-vector pairs $(f_0,f_1)$ of $4$-polytopes is equal to
\begin{align*}
\PiFset{4}{0,1}{}= \ \{&(f_0,f_1)\in \Z^2:
10 \le 2f_0 \le f_1 \le \tfrac{1}{2} f_0(f_0 - 1)\}\\
&\backslash \{(6,12),(7,14),(8,17), (10,20)\}.
\end{align*}
\end{theorem}

The existence parts of Theorems~\ref{thm:f0f3} and \ref{thm:f0f1} are proved
by taking neighborly polytopes, which yield the polytopal pairs on the upper bound,
as well as dual neighborly polytopes for the polytopal pairs on the lower bound,
and by finding some polytopes for examples of small polyhedral pairs.
From these polytopes, polytopes with all other possible polytopal pairs 
are constructed by an inductive method of (generalized) stacking 
(see Sections~\ref{sec:stacking} and \ref{sec:genstacking}).
\begin{theorem}[Barnette \& Reay {\cite[Thm.~10]{BarnetteReay}}] \label{thm:f0f2}
The set of $f$-vector pairs $(f_0,f_2)$ of $4$-polytopes is equal to
\begin{align*}
\PiFset{4}{0,2}{}=\{(f_0,f_2)\in \Z^2: \
10\le&\tfrac12(2f_0 + 3 + \sqrt{8f_0 + 9}) \le f_2 \le f_0^2 - 3f_0,\\
&f_2 \ne f_0^2 - 3f_0 -1
\}\\
\backslash \{&(6,12),(6,14),(7,13),(7,15),(8,15),\\
&(8,16),(9,16),(10,17),(11,20), (13,21)\}.
\end{align*}
\end{theorem}

The existence part of Theorem~\ref{thm:f0f2} is proved similarly to the proofs of Theorems~\ref{thm:f0f3} and \ref{thm:f0f1},
additionally considering all $4$-dimensional pyramids, bipyramids and prisms (``cylinders''). 
\begin{theorem}[Barnette {\cite[Thm.~1]{Barnette}}, with corrections, cf.\ \cite{HoppnerGMZ}] \label{thm:f1f2} 
The set of $f$-vector pairs $(f_1,f_2)$ of $4$-polytopes is equal to
\begin{alignat*}{3}
\PiFset{4}{1,2}{}=\{&(f_1,f_2)\in \Z^2:
&&10 \le \tfrac1 2 f_1 +\Big \lceil\sqrt{f_1+\tfrac 9 4}+\tfrac 1 2\Big \rceil+1\le f_2 ,\\
&&&10 \le \tfrac1 2 f_2 +\Big \lceil\sqrt{f_2+\tfrac 9 4}+\tfrac 1 2\Big \rceil+1\le f_1,\\
&&&f_2\ne \tfrac1 2 f_1 +\sqrt{f_1+\tfrac {13} 4}+2,\\
&&&f_1\ne \tfrac1 2 f_2 +\sqrt{f_2+\tfrac {13} 4}+2\}\\
\backslash\{&(12,12),(13,14),&&(14,13),(14,14),(15,15),\\
&(15,16),(16,15),&&(16,17),(16,18),(17,16),\\
&(17,20),(18,16),&&(18,18),(19,21),(20,17),(20,23),\\
&(20,24),(21,19),&&(21,26),(23,20),(24,20),(26,21)\}.
\end{alignat*}
\end{theorem}

For the proof of Theorem~\ref{thm:f1f2}, a finite number of  polytopes with few edges was found, and polytopes 
with all other possible polytopal pairs were constructed using an inductive method based on ``facet splitting''
(see  Section~\ref{sec:facetsplitting}).

The remaining $f$-vector projections are given by duality.

\begin{figure*}[t!]
    \centering
    \begin{subfigure}[t]{0.5\textwidth}
        \centering
        \includegraphics[height=2in]{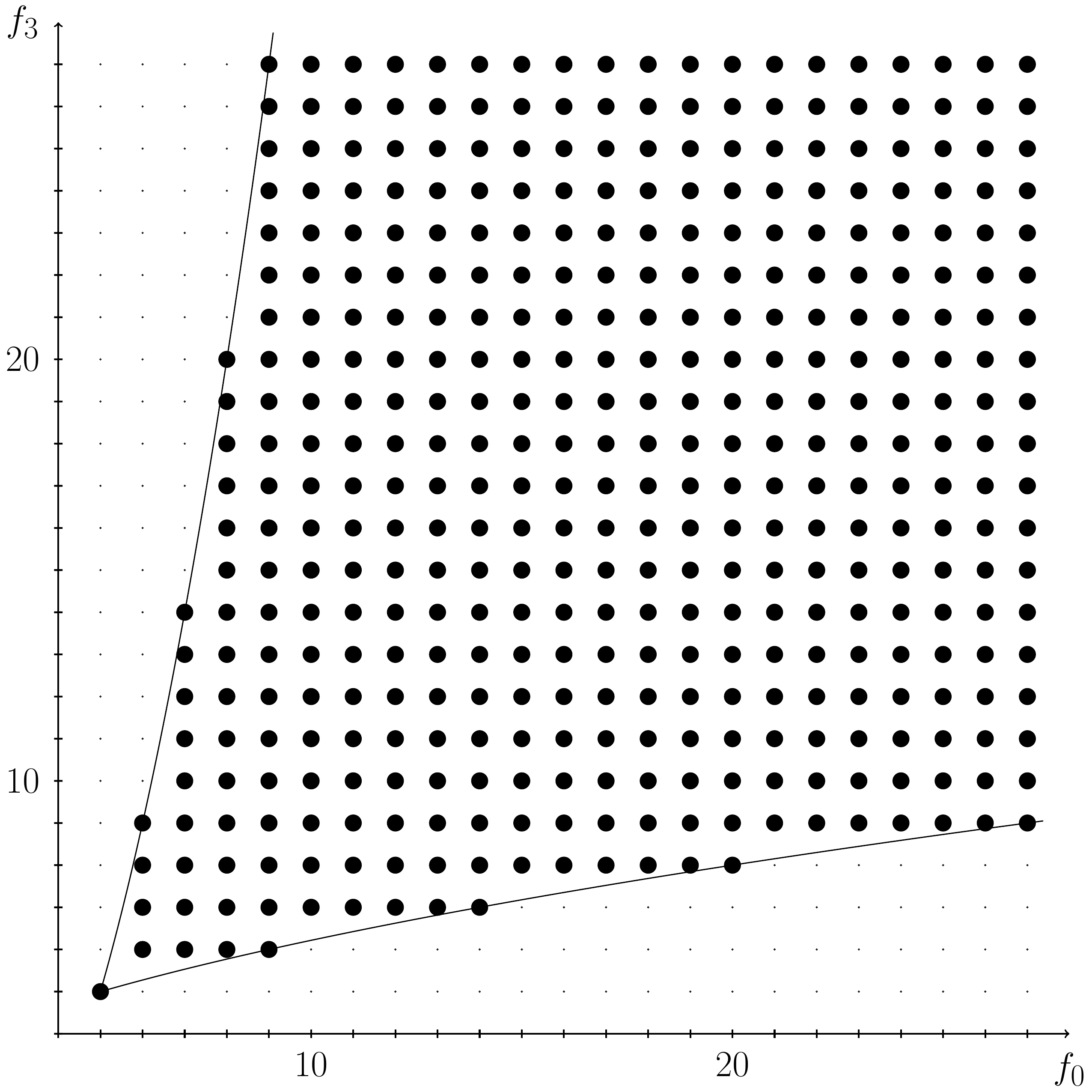}
        \caption{$\PiFset{4}{0,3}{}$}
    \end{subfigure}%
    ~ 
    \begin{subfigure}[t]{0.5\textwidth}
        \centering
        \includegraphics[height=2in]{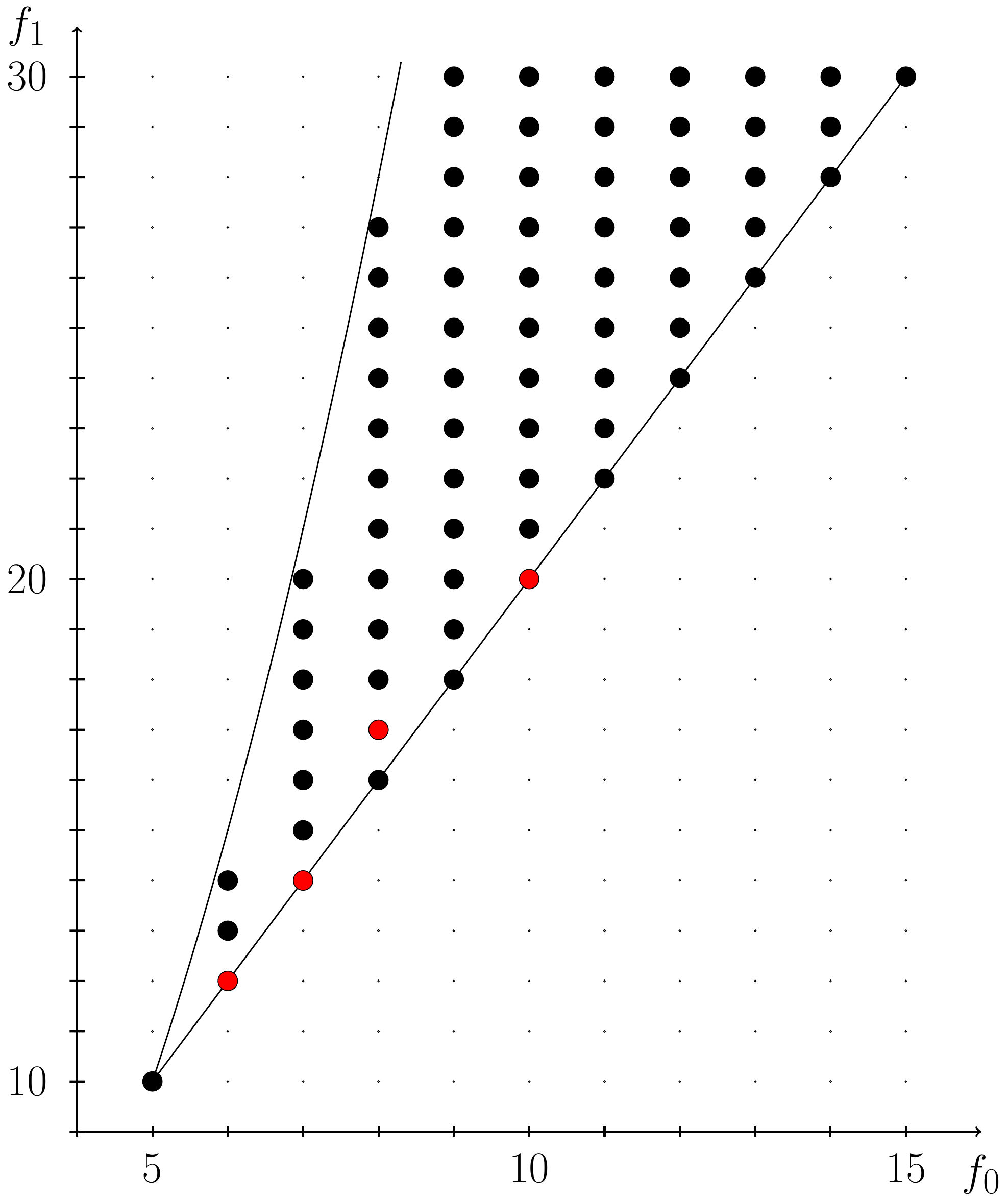}
        \caption{$\PiFset{4}{0,1}{}$}
    \end{subfigure}
    \\
    \begin{subfigure}[t]{0.5\textwidth}
      \centering
      \includegraphics[height=2in]{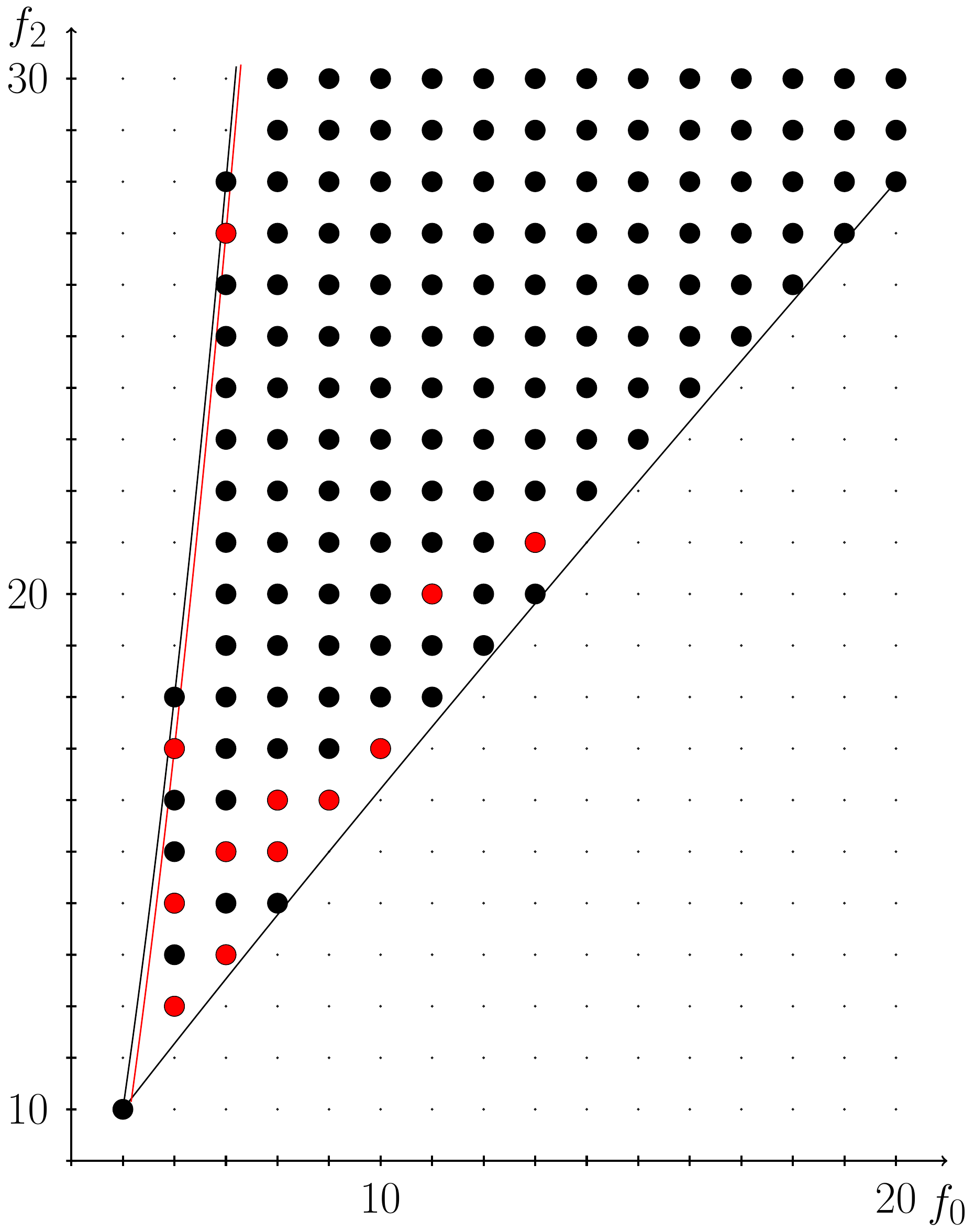}
      \caption{$\PiFset{4}{0,2}{}$}
    \end{subfigure}%
    ~ 
    \begin{subfigure}[t]{0.5\textwidth}
        \centering
        \includegraphics[height=2in]{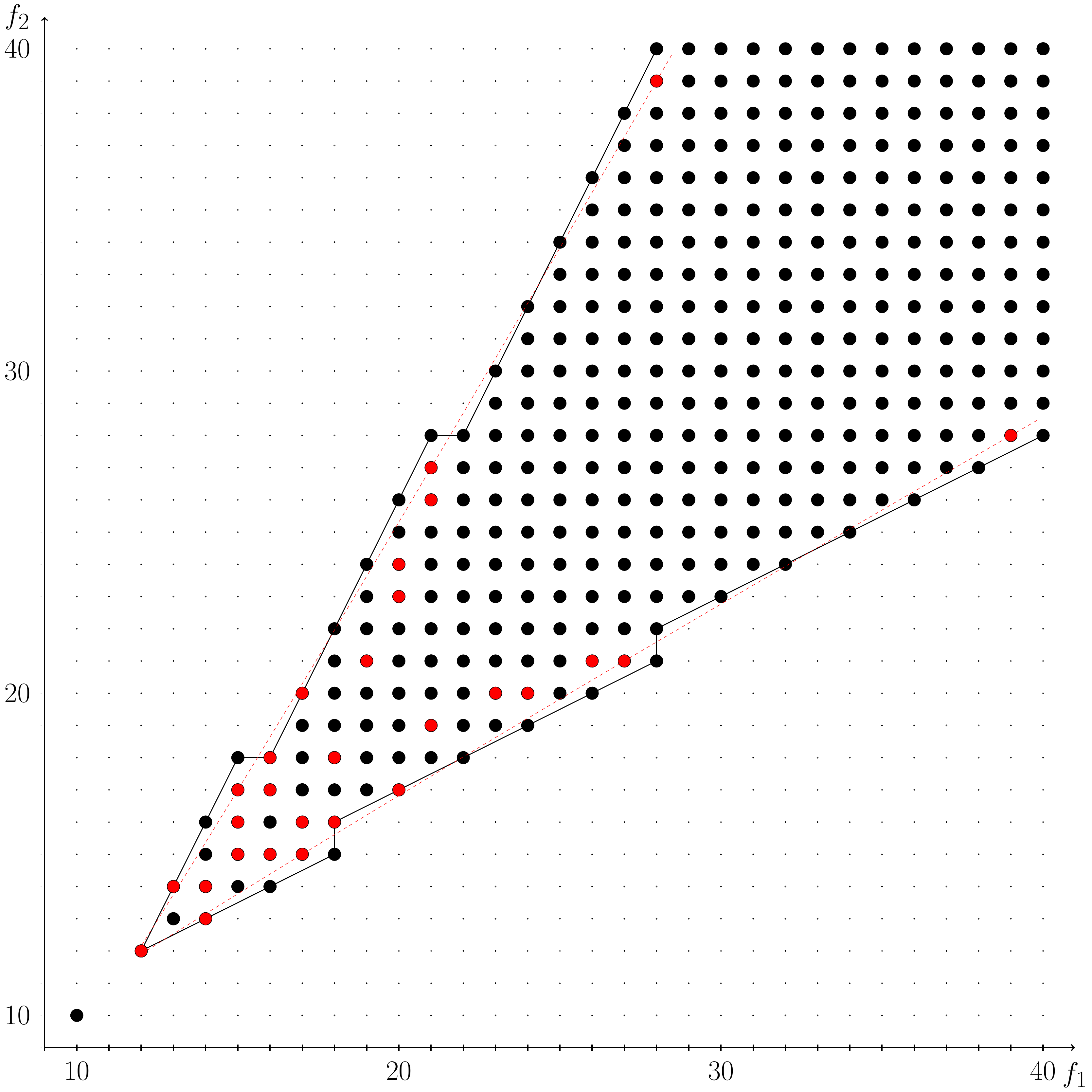}
        \caption{$\PiFset{4}{1,2}{}$}
    \end{subfigure}
    \caption{$f$-vector projections, red dots are exceptional pairs}
\label{plot:fvec}
\end{figure*}
\subsection{Flag vector pair $(f_0,f_{03})$ for $4$-polytopes}
In the following we will characterize the set
\[
	\PiFset{4}{0,03}{}=\{(f_0(P),f_{03}(P)) \in \Z^2 \ \vert \ P \text{ is a $4$-polytope}\},
\]
that is, we describe the possible number of vertex-facet incidences of a $4$-polytope
with a fixed number of vertices.
Equivalently, this tells us the possible average number of facets of the vertex figures, $\frac{f_{03}}{f_0}$,
for a given number $f_0$ of vertices.

 In 1984  Altshuler and Steinberg classified all combinatorial types of 
$4$-polytopes with up to $8$ vertices~\cite{AS84,AS85}.
This classification makes our proof much easier. 
We will use the classification to find examples of polytopes for certain small polytopal pairs
and also to argue that some pairs cannot be polytopal pairs of any $4$-polytope.
The following  is our first main theorem:

 \begin{theorem}\label{main_thm1}
There exists a $4$-polytope $P$ with $f_0(P) = f_0$ and $f_{03}(P) = f_{03}$ 
if and only if $f_0$ and $f_{03}$ are integers satisfying
\begin{align*}
20 \le 4f_0 \le \ &f_{03} \le 2f_0(f_0 - 3),\\
&f_{03} \ne 2f_0(f_0-3)-k \text{ for } k \in \{1,2,3,5,6,9,13\}
\end{align*}
 and $(f_0,f_{03})$ is not one of the $18$ exceptional pairs
\begin{align*}
&(6,24),(6,25),(6,28),\\
& (7,28), (7,30), (7,31),\\ 
&(7,33), (7,34), (7,37), (7,40),\\
&(8,33), (8,34),(8,37), (8,40),\\
& (9,37), (9,40),
(10,40), (10,43).
 \end{align*}
\end{theorem}

See Figure~\ref{plot:f0f03} for a visualization of the projection in the plane $(f_0,f_{03}-4f_0)$.
\begin{figure}[htbp]
  \centering
\includegraphics[width=10.5cm]{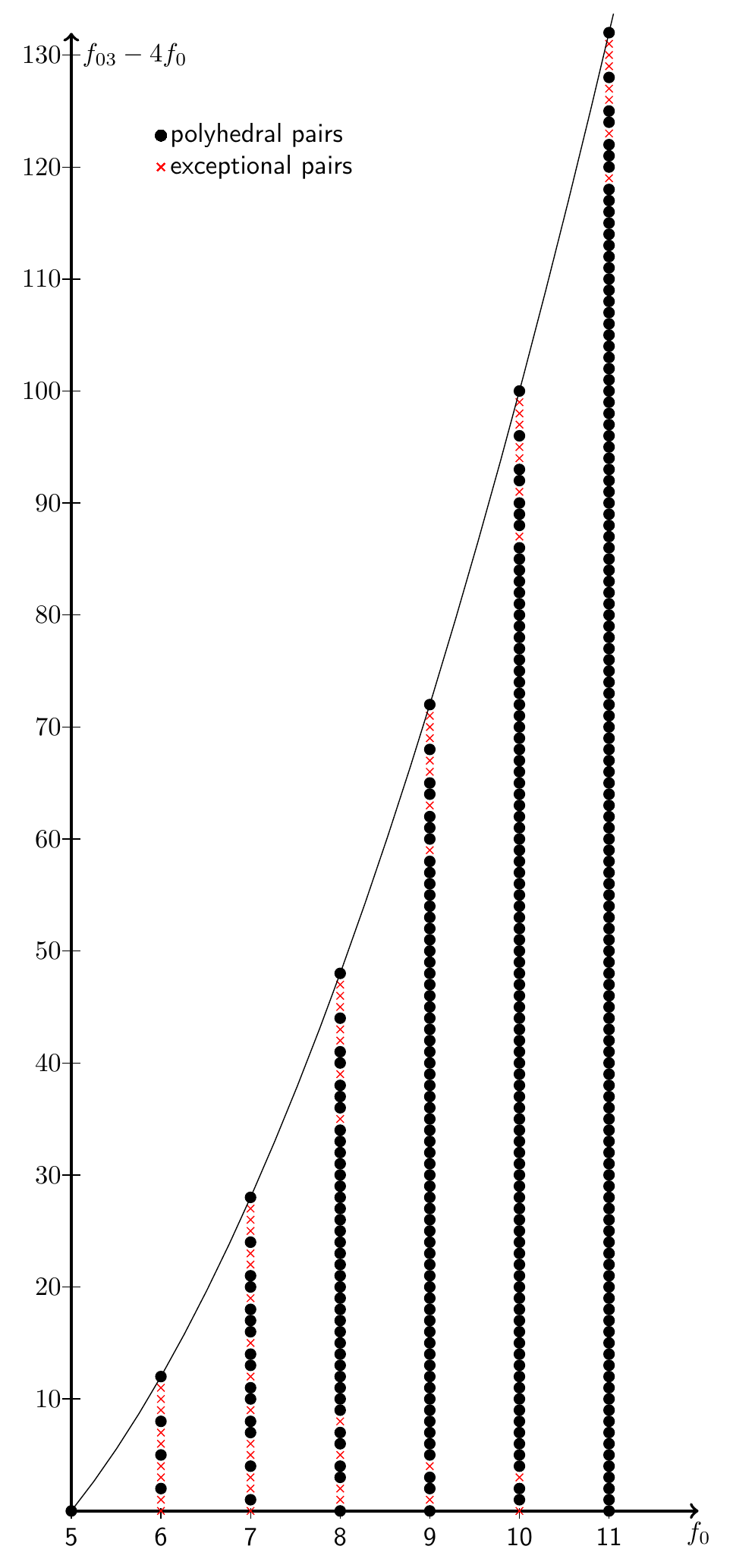}
\caption{Projection $\PiFset{4}{0,03}{}$}
\label{plot:f0f03}
\end{figure}

The proof of Theorem~\ref{main_thm1} follows the proofs of the projections of the 
$f$-vector~\cite{Barnette}, \cite{BarnetteReay}, \cite{Grunbaum},
by taking small polytopal pairs as well as polytopal pairs on the boundaries and constructing 
new polytopal pairs from the given ones.
The inductive methods used for this proof are the  stacking and truncating constructions
from Theorem~\ref{thm:f0f3}, \ref{thm:f0f1} and \ref{thm:f0f2} and ``facet splitting'' methods 
generalized from the methods used in the proof of Theorem~\ref{thm:f1f2}.
 
\begin{lemma}\label{bounds}
If $P$ is a $4$-dimensional polytope with $f_0$ vertices and $ f_{03} $ vertex-facet incidences, then
 \[
 	 4f_0 \le f_{03} \le 2f_0(f_0 -3). 
 \]
 \end{lemma} 
 \begin{proof}
Every vertex of a $d$-polytope lies in at least $d$ facets, so
clearly $4f_0 \le f_{03}$ holds for all $4$-dimensional polytopes,  with equality if and only if $P$ is simple.
 
The second inequality follows from a generalization of the upper bound theorem to flag vectors:
For any $d$-dimensional polytope with $n$ vertices and for any $S \subseteq \{0, \dots,d-1\}$,
\[
	f_S \le f_S(C_d(n)),
\]
where $C_d(n)$ is the $d$-dimensional cyclic polytope with $n$ vertices \cite[Thm~18.5.9]{BilleraBjorner}.
In particular, $4$-dimensional cyclic polytopes are simplicial, and for any $4$-dimensional polytope~$P$,
 \[
 	f_{03}(P)\le f_{03}(C_4(n))=4f_3(C_4(n))=2n(n -3)
 \]
with equality if and only if $P$ is neighborly.
\qed
\end{proof}
  
\begin{lemma}\label{thm:exceptions}
There is no $4$-polytope $P$ with $f_0(P) =f_0$ and $f_{03}(P) = f_{03}$
if $(f_0,f_{03})$ is any of the following pairs:
 \begin{align*}
&(6,24),(6,25),(6,28),\\
& (7,28), (7,30), (7,31), \\
&(7,33), (7,34), (7,37), (7,40),\\
&(8,33), (8,34),(8,37), (8,40),\\
& (9,37), (9,40), (10,40), (10,43),\\
&(f_0,2f_0(f_0-3)-k) \ \text{ for } k\in \{1,2,3,5,6,9,13\} \text{ and for any } f_0\ge 6.
 \end{align*}
\end{lemma}

For the proof of this lemma we need some equations and inequalities which hold for the flag vector 
of any $4$-polytope.
\begin{lemma}[Generalized Dehn--Sommerville equations, Bayer \& Billera 
{\cite[Thm.~2.1]{BayerBillera}}]\label{genDehnSommerville}  
Let $P$ be a $d$-polytope and $S \subseteq \{0, 1, \dots , d-1\}$.
Let $\{i,k\} \subseteq S \cup \{-1,d\}$ such that  $i < k-1$
and such that there is no $j \in S$ for which $i < j < k$. Then
\[
	\sum\limits_{j=i+1}^{k-1} (-1)^{j-i-1}f_{S \cup j}(P) = f_S(P)(1-(-1)^{k-i-1}).
\]
\end{lemma}

For $d=4$, $S= \{0\}$,  $i=0$, $k =4$ and with the observation $f_{01} = 2f_1$ we obtain
\begin{align}
f_{02} &= -2f_0 +2f_1 +f_{03}.\label{GDS_equ}
\end{align}

\begin{lemma}[Bayer {\cite[Thm.~1.3, 1.4]{Bayer}}]\label{facets:upper_bound}
The flag vector of every $4$-polytope satisfies the inequalities
\begin{align}
f_{02} -3f_2+ f_1  - 4f_0 +10 &\ge  0 \\
\text{ and } -6 f_0 +6f_1- f_{02} &\ge 0. \label{centerboolean}
\end{align}
\end{lemma}

Inequality~\textnormal{(\ref{centerboolean})} holds with equality if and only if  the $4$-polytope 
is center boolean, that is, if all its facets are simple.

Using Lemma~\ref{genDehnSommerville} and the Euler--Poincar\'{e} formula for dimension $4$
{\cite[Thm.~8.1.1]{Grunbaum}} to rewrite Lemma~\ref{facets:upper_bound}, we obtain the inequalities
\begin{align}
-3f_0 -3f_3 + f_{03} +10 &\ge 0\label{facet_upper} \\
\text{ and } 4f_0 -4f_1 +f_{03} &\le 0. \label{edge_lower}
\end{align}
\begin{proof}[Proof of Lemma \ref{thm:exceptions}]
We first show that there is no polytope $P$ with 
\[
	(f_0(P),f_{03}(P)) = (f_0, 2f_0(f_0-3) -k) \ \text{ for } k \in \{1,2,3,5,6,9,13\}.
\]
For $k=1,2,3$ we prove the non-existence directly. 
For $k=5,6$ we show that if $P$ is a  polytope with $2f_0(f_0-3) -5 \le f_{03} \le  2f_0(f_0-3) -7$, 
then necessarily $f_{03} =  2f_0(f_0-3) -7$.
The proof for $k=9$ and $k=13$ follows similarly.

For $k > 0$ any $4$-polytope with polytopal pair $(f_0, 2f_0(f_0-3) -k)$ cannot be neighborly, so
\[
	f_1 < \binom{f_0}{2}.
\]
On the other hand, for $(f_0(P),f_{03}(P)) = (f_0, 2f_0(f_0-3) -k)$ 
Inequality~(\ref{edge_lower}) reads\\ $\frac{1}{2}f_0(f_0-1)-\frac k4 \le f_1$.
Both inequalities together give
\begin{align}
 \binom{f_0}{2} - \frac k4 \le f_1 <  \binom{f_0}{2}. \label{edge_equation}
\end{align}
There is no integer solution for $k=1,2,3$.
 
For $k=5,6,7$, the only possible integer value for $f_1$ is $\frac{1}{2}f_0(f_0-1)-1$.
Assume that $P$ is a polytope with
\[
	f_1 = \binom{f_0}{2} -1
\] 
and
\[
	2f_0(f_0-3) -5 \le f_{03} \le  2f_0(f_0-3) -7.
\]
Since $f_1 = \binom{f_0}{2} -1$, there is a unique pair $v_1,v_2$ of vertices of $P$ 
not forming an edge. We call such a pair of vertices a \emph{non-edge}.
Any facet of $P$ which is not a simplex must contain this non-edge, since the only 
$3$-polytope in which every two vertices form an edge is the simplex.
Consider a facet $F$ which is not a simplex, and therefore contains the unique non-edge.
Such a facet $F$ needs to exist, since if $P$ were simplicial, $f_{03}\equiv 0 \mod 4$.
Observe that if $F$ would have more than five vertices, then we could find five vertices 
of $F$ for which every two vertices form an edge.
This subpolytope of $F$ could not be $d$-dimensional, for $d \le 3$. 
From this contradiction follows that $F$ has five vertices.
The only combinatorial types of $3$-polytopes with five vertices are the square pyramid 
and the bipyramid over a triangle, only the latter has exactly one non-edge. 
So $F$ is a bipyramid, and the non-edge is between the apices of $F$.
If there were another non-tetrahedral facet of $P$, it would intersect $F$ in a common face containing the non-edge.
Such a face does not exist, and hence $P$ is a polytope with one bipyramidal facet and $t$ 
tetrahedral facets, for some integer $t$. 
This implies that
 \[
	 f_{03} = 4t+5 \equiv 1 \mod 4.
 \]
From the assumption $ 2f_0(f_0-3) -5 \le f_{03} \le  2f_0(f_0-3) -7$ follows now 
 \[ 
 	f_{03} =  2f_0(f_0-3) -7.
 \]
Assume now that there is a polytope $P$ with 
 \[
	(f_0(P),f_{03}(P))= (f_0, 2f_0(f_0-3) -9).
 \] 
Inequality~(\ref{edge_equation}) implies that
 \[
	 f_1= \binom{f_0}{2} -2 \text{ or }  f_1= \binom{f_0}{2} -1.
 \]
If $f_1= \binom{f_0}{2} -1$, then we have just proved that $f_{03} \equiv 1 \mod 4$. 
Since $f_{03}(P) \equiv 3 \mod 4$, it follows that $P$ has two non-edges. 
The inequality $f_1 \le 3f_0 -6$ holds for $3$-dimensional polytopes and 
any facet $F$ has at most two non-edges:
 \begin{align*}
 \binom{f_0(F)}{2} -2\le f_1(F)\le 3f_0(F)-6 \Rightarrow f_0(F)<6.
 \end{align*}
Any non-tetrahedral facet is hence a polytope with five vertices, a bipyramid over a triangle or a square pyramid.
Since $f_{03}(P) \equiv 3 \mod 4$, there have to be at least three non-tetrahedral facets. 
Bipyramids have one non-edge, not contained in any other facet.
Square pyramids have two non-edges, which are both contained in exactly one other facet.
This contradicts the fact that there are only two non-edges in $P$.
In conclusion, there is no polytope with 
$(f_0,f_{03})$ $= (f_0, 2f_0(f_0-3) -9)$.
   
Finally, assume that there exists a polytope $P$ with polytopal pair 
\[
	(f_0(P),f_{03}(P))= (f_0, 2f_0(f_0-3) -13).
\] 
From Inequality~(\ref{edge_equation}) it follows that $P$ has 
$\binom{f_0}{2} -3$, $\binom{f_0}{2} -2$ or $\binom{f_0}{2} -1$ edges.
Each facet $F$ of $P$ has at most three non-edges.
 For any facet $F$ of $P$ the inequality
$f_1(F) \le 3f_0(F) -6$ now yields
$\binom{f_0(F)}{2} -3 \le 3f_0(F)-6$
$\Rightarrow f_0(F) <7$.
If $F$ has six vertices, it must have $12$ edges and three non-edges.
There are only two such combinatorially different $3$-polytopes, which both are simplicial.
 
Assume that $P$ has a facet $F$ with six vertices. Then $F$ contains three non-edges,
all of them not in any $2$-face of $F$ and hence not in any other facet.
So all other facets of $P$ are tetrahedra.
This is a contradiction to $f_{03}(P) \equiv 3 \mod 4$.

$P$ is not simplicial, so there are non-tetrahedral facets, all of them with five vertices. 
Observe that since there are at most three non-edges, we cannot have more than three non-tetrahedral facets.
Together with $f_{03}(P) \equiv 3 \mod 4$, this leaves us with two cases:
\begin{align*}
 &\text{(i) The non-tetrahedral facets of $P$ are three bipyramids over triangles.} \\
 &\text{(ii) The non-tetrahedral facets of $P$ are two square pyramids and one}\\
 &\text{bipyramid over a triangle.} 
\end{align*}
In both cases,  let $t$ denote the number of tetrahedra in $P$.
Then
\begin{align*}
f_{03}(P)&= 2f_0(f_0-3) -13= 4t +3 \cdot 5\\
\Rightarrow \  t &= \tfrac{1}{2}f_0(f_0-3)-7\\
\Rightarrow \  f_3(P) &= t+3 = \tfrac{1}{2}f_0(f_0-3)-4.
\end{align*}
We can now calculate $f_2(P)$ in two ways.
From the  Euler--Poincar\'e formula,
\begin{align*}
 f_2&=f_1+f_3-f_0 \\
 &= \binom{f_0}{2} -3 + \tfrac{1}{2}f_0(f_0-3)-4 - f_0\\
 &= f_0(f_0-3)-7.
\end{align*}
Each $2$-face lies in exactly two facets.
The number of $2$-faces of $P$ can therefore also be calculated
by counting the number of $2$-faces in each facet.
In case (i) this gives:
\begin{align*}
&f_2= \tfrac{1}{2}f_{23}
 = \tfrac{1}{2}(4t + 3 \cdot 6)
 = f_0(f_0-3)-5
 \ne  f_0(f_0-3)-7.
\intertext{In case (ii) we obtain:}
&f_2= \tfrac{1}{2}f_{23}
 = \tfrac{1}{2}(4t + 2 \cdot 5 + 6)
 = f_0(f_0-3)-6
 \ne  f_0(f_0-3)-7.
\end{align*}
So there cannot be a polytope with polytopal pair $(f_0, 2f_0(f_0-3) -13)$.
\smallskip

It remains to show the non-existence of $18$ pairs $(f_0, f_{03})$.
All combinatorial types of $4$-polytopes with up to $8$ vertices have been classified by
Altshuler and Steinberg~\cite{AS84,AS85}.
From this classification it follows that there are no polytopes with polytopal pairs
\begin{align*}
&(6,24),(6,25),(6,28),\\
& (7,28), (7,30), (7,31), \\
&(7,33), (7,34), (7,37), (7,40), \\
&(8,33), (8,34),(8,37) \text{ or } (8,40).
\end{align*}
To see that the four pairs $(9,37)$, $(9,40)$, $(10,40)$ and $(10,43)$ are exceptional pairs,
we make use of the  upper bound for the number of facets in terms of the number of vertices and 
vertex-facet incidences.
If there were a polytope $P$ with polytopal pair
$(9,37)$, $(9,40)$, $(10,40)$ or $(10,43)$,
due to Inequality~(\ref{facet_upper}) it would need to have less than $8$ facets.
By duality, this would give us a polytope $P^*$ with 
$f_{03}(P^*)  = 37,40$ or $43$ and $f_0(P^*)  \le 7$.
From the upper bound $f_{03} \le 2f_0(f_0-3)$ it follows that $f_0(P^*) = 7$.
As seen above, polytopes with polytopal pair $(7,37)$ or $(7,40)$ do not appear in the classification.
Pair $(7,43)$ is of the type $(f_0,2f_0(f_0-3)-13)$, which is an exceptional pair.
\qed
\end{proof}

We will use the classification of $4$-dimensional polytopes with up to $8$ vertices~\cite{AS84,AS85}
together with some classes of polytopes, such as cyclic polytopes, pyramids, and some additional 
polytopes, and from those polytopes and their polytopal pairs construct all other possible polytopal pairs. 
The methods needed for this construction are described in the following sections.

\subsection{Stacking and truncating}\label{sec:stacking}
The operations \emph{stacking} and \emph{truncating} (see \cite[Sect.~16.2.1]{HenkRGGMZ}) 
turn out to be essential in finding examples of polytopes for all possible polytopal pairs $(f_0,f_{03})$.
Let $P$ be a $4$-polytope with at least one simplex facet $F$ 
and $v$ a point beyond $F$ and beneath all other facets of $P$. 
Let $Q = \textnormal{conv}(\{v\} \cup P)$.
Then 
\[
	f_0(Q) = f_0(P) +1 \text{ and } f_{03}(Q) = f_{03}(P) +12.
\]
Dually, let $Q$ be a polytope obtained by truncating a simple vertex from a polytope $P$.
Then 
\[
	f_0(Q) = f_0(P)+3 \text{ and } f_{03}(Q) = f_{03}(P)+12.
\]
The polytopes obtained through these two methods all have both a simple vertex and a simplex facet. 
This means that we can stack vertices on simplex facets and truncate simple vertices repeatedly. 
Truncating simple vertices and stacking vertices on simplex facets inductively, 
starting from a polytope with $ (f_0, f_{03} )$ with tetrahedral facet and simple vertex, 
we obtain new polytopes with
\[ 
	(f_0 + 2m + n, f_{03} +12n) \ \text{ for } n \ge 0, \ \ 0 \le m \le n.
\]
 Given a  polytope $P$ with a square pyramidal facet $F$, let $v$ be a point 
beyond $F$ and beneath all other facets of $P$. 
Let $Q = \conv(\{v\} \cup P$).
Then 
\[
	f_0(Q) = f_0(P) +1 \text{ and } f_{03}(Q) = f_{03}(P) +16.
\]
The results in this section are simple consequences from
\cite[Thm.~5.2.1]{Grunbaum}, with corrections from \cite{AltshulerShemer}.
\subsection{Generalized stacking on cyclic polytopes}\label{sec:genstacking} 
We need some more methods, especially to create polytopes with polytopal pair 
$(f_0, f_{03})$ close to the upper bound $f_{03}=2f_0(f_0-3)$.
For our next construction we need the observation that every cyclic $4$-polytope 
with $n$ vertices has edges that lie in exactly $n-2$ facets.
Such edges are called \emph{universal edges}.
The following construction was used by Grünbaum~\cite[Sect.~10.4.1]{Grunbaum}
for the characterization of the 
sets $\PiFset{4}{0,3}{}$ and $\PiFset{4}{0,1}{}$.
Starting from a cyclic polytope with $n$ vertices, we can obtain new polytopes by 
stacking a vertex onto it, such that the vertex lies beyond several facets. 
Let $R_i(n), i \in \{1,\dots ,n - 3\}$, denote a polytope obtained from 
the cyclic polytope $C_4(n)$ with $n$ vertices as the convex hull of $C_4(n)$ and
a point $v$, where $v$ is beyond $i$ facets of $C_4(n)$ sharing a universal edge.
Let $F_1, \dots , F_i$ denote these $i$ facets, such that $F_j$ and $F_{j+1}$ 
meet in a common $2$-face, for all $j = 1, \dots , i-1$.
Then  the new polytope $R_i(n)$ has one more vertex than $C_4(n)$ and the following facets:\\

 \noindent 
(1) All $\frac{1}{2}n(n-3) -i$ facets of $C_4(n)$ which $v$ lies beneath.
\\

\noindent
(2) Facets which are convex hulls of $v$
and $2$-faces of $C_4(n)$ that are contained 
in both a facet which $v$ is beyond and a facet which $v$ is beneath.
There are two types of these facets:
\\

(2a) Two such facets for each of the $i-2$ facets 
$F_2, \dots , F_{i-1}$
which $v$ lies beyond 
and which share two $2$-faces with other facets which $v$ lies beyond. 
\\

(2b) Three new facets for each of the two facets $F_1$ and $F_i$ 
which $v$ lies beyond and which share one $2$-face with other facets which $v$ lies beyond.
\\

Note that all these facets are simplices.
 In conclusion, for $1\le i\le n-3$,
\[
	f_0(R_i(n)) = n + 1 \text{ and } f_{03}(R_i(n)) = 2n(n-3) +4i + 8.
\]
Observe that
\[
	f_{03}(C_4(n+1)) = f_{03}(C_4(n)) +4n-4,
\]
so if $i = n-3$, we obtain again a neighborly polytope, with $n+1$ vertices.

\subsection{Facet splitting}\label{sec:facetsplitting}
We need to generalize the stacking method even more to obtain non-simplicial polytopes,
Compare to the \emph{$A$-sewing construction} of Lee and Menzel~\cite{LM10}.
For easier visualization, we choose to work in the dual setting.
Instead of adding a new vertex to a polytope, we will create a new facet in the dual polytope.
This method of \emph{facet splitting} was used by Barnette~\cite{Barnette}
for the classification of $\PiFset{4}{1,2}{}$:
Consider a facet $F$ of a $4$-polytope $P$ and a hyperplane $H$ which intersects the relative interior of $F$ 
in a  polygon $X$. If on one side of $H$, the only vertices of $P$ are simple vertices of $F$,
then we can obtain a new polytope $P^{\prime}$ by separating facet $F$ 
into two new facets by the polygon $X$.
We say that $P^{\prime}$ is obtained from $P$ by \emph{facet splitting}.
      
\subsubsection{Dual of a cyclic polytope}\label{sec:dualcyclic}
    
We will split a facet of the dual of a cyclic polytope (see Barnette~\cite{Barnette}).
$C_4^{*}(n)$, the dual of the cyclic polytope with $n$ vertices,
is a simple polytope with $n$ facets, each facet having $2(n-3)$ vertices.
The facets are all \emph{wedges} over $(n-2)$-gons, that is, polytopes 
with two triangular $2$-faces, $n-5$ quadrilateral $2$-faces and two $(n-2)$-gons meeting in an edge.
Let $G$ be a $2$-dimensional plane in the affine hull of a facet $F$ of $C_4^{*}(n)$. 
Let $X$ be the intersection of $F$ and $G$. All vertices of  $C_4^{*}(n)$
are simple, so we can obtain a new polytope by facet splitting of $C_4^{*}(n)$
by choosing a hyperplane $H$ which contains $G$ such that on one side of $H$
the only vertices of $C_4^{*}(n)$ are vertices of $F$. 
Such a hyperplane can be found by taking
the facet-defining hyperplane of $F$ and rotating it about $G$.
The combinatorial properties $f_1$ and $f_{03}$ of the polytope obtained 
through facet splitting depend on the choice of $G$:
We can choose $G$ not to intersect any vertices of $F$.
Then, for any $i$ such that $3 \le i \le n-2$, 
$X = G \cap F$ can be chosen to be an $i$-gon.
Let $\delta _0(i,n)$ denote the polytope obtained through facet splitting 
for this choice of $G$ (see Figure~\ref{plot:wedge_delta0}).
Now $\delta _0(i,n)$ has one more facet and $i$ more vertices than $C_4^{*}(n)$. 
As $C_4^{*}(n)$ is a simple polytope, all of its edges lie in exactly three facets and
each of the $i$ new vertices of $\delta _0(i,n)$ lies in four facets.
The new polytope has therefore $4i$ more vertex-facet incidences than $C_4^{*}(n)$.
If we instead choose $G$ to intersect exactly one vertex of $F$,
$X$ can again be any $i$-gon for $3 \le i \le n-2$.
Call this polytope $\delta_1(i,n)$. It has  one more facet and  $i -1$ more vertices than $C_4^{*}(n)$.
The $i -1$ new vertices are simple, and the one vertex of $C_4^{*}(n)$ which lies in $X$
is contained in one additional facet.
In total, $f_{03}$ increases by $4i-3$.
The polytopes $\delta_0 (i,n)$ and $\delta_1 (i,n)$ are used in the characterization of 
$\PiFset{4}{1,2}{}$~\cite{Barnette}.

Similarly, let $\delta_2(i,n)$ denote the polytope obtained when $G$ intersects two vertices of $F$.
As before, $i$ can be chosen to be any integer between $3$ and $n-2$.
The new polytope has one more facet, $i-2$ more vertices and $4i-6$ more vertex-facet incidences.
If we choose $G$ to intersect $F$ in three vertices, as the intersection of $G$ and $F$ 
we can obtain $i$-gons for $3 \le i \le n-3$ (see Figure~\ref{plot:wedge_delta3}).
The new polytope, denoted by $\delta_3(i,n)$, has one more facet, $i-3$ more vertices and $4i-9$ more 
vertex-facet incidences.
Let us look at the duals of these polytopes.  
For $ 3 \le i \le n-2$ we obtain polytopes $\delta^*_0(i,n)$, $\delta^*_1(i,n)$ and $\delta^*_2(i,n)$ with
\begin{align*}
    &(f_0(\delta^*_0(i,n)),f_{03}(\delta^*_0  (i,n)))=(n+1, 2n(n-3) +4i),\\ 
    &(f_0(\delta^* _1(i,n)),f_{03}(\delta^*_1(i,n)))=(n+1, 2n(n-3) +4i-3),\\
    &(f_0(\delta^* _2(i,n)),f_{03}(\delta^*_2(i,n)))=(n+1, 2n(n-3) +4i-6).
 \end{align*}
For $3 \le i \le n-3$ we obtain polytopes $\delta^*_3(i,n)$ with
\begin{align*}
     &(f_0(\delta^* _3(i,n)),f_{03}(\delta^*_3(i,n)))=(n+1, 2n(n-3) +4i-9).
\end{align*}
In particular, the polytopes $\delta^*_0(i,n)$, $\delta^*_1(i,n)$, $\delta^*_2(i,n)$ and $\delta^*_3(i,n)$
have simplex facets.
\begin{figure}
     \centering
     \begin{subfigure}{.5\textwidth}
       \centering
       \includegraphics[width=\linewidth]{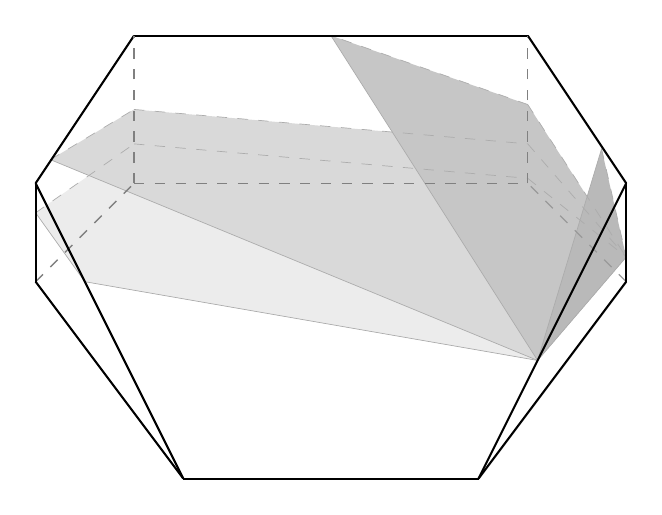}
       \caption{$\delta _0(i,8)$}
       \label{plot:wedge_delta0}
     \end{subfigure}%
     \begin{subfigure}{.5\textwidth}
         \centering
         \includegraphics[width=\linewidth]{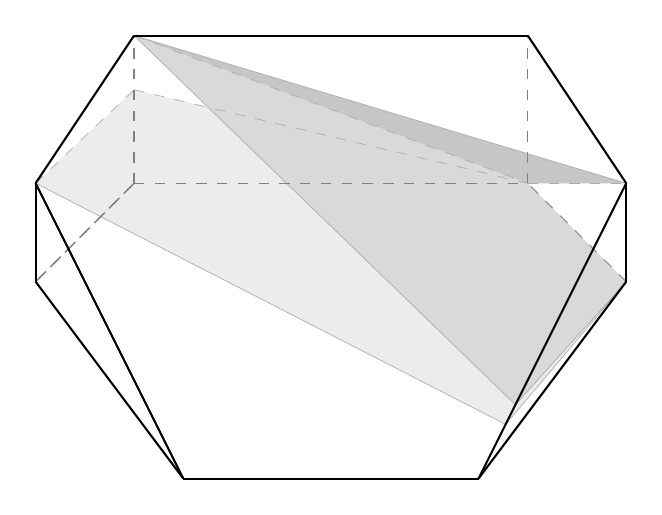}
         \caption{$\delta _3(i,8)$}
         \label{plot:wedge_delta3}
     \end{subfigure}
     \caption{Facet of $C_4^{*}(8)$ split by an $i$-gon}
\end{figure}
 
\subsubsection{Polytopes with a bipyramidal facet}\label{sec:bipyramid}
Given a polytope $P$ with a facet $B$ which is a bipyramid over a triangle,
such that at least one apex  $v$ of $B$ is a simple vertex,
we can split the bipyramid into two tetrahedra by ``moving'' $v$ 
outside the affine hull of $B$, along the unique edge 
which contains $v$ and does not belong to $B$.
The new polytope $\widetilde{P}$ has the same number of vertices and one more facet than $P$.
The apices of the bipyramid still belong to the same number of facets as before,
but the other three vertices now belong to one more facet.
In total, the number of vertex-facet incidences increases by~$3$.
Hence,
\[
	(f_0(\widetilde{P}),f_{03}(\widetilde{P}))=(f_0(P),f_{03}(P) +3).
\]
Note that $\widetilde{P}$ has simplex facets and that
any simple vertex of $P$ is a simple vertex of $\widetilde{P}$.

\subsection{Construction of polytopal pairs $(f_0,f_{03})$}\label{sec:construction}
We can now prove Theorem~\ref{main_thm1}.
First, we list some examples of polytopes with small polytopal pairs $(f_0,f_{03})$ 
for $f_{03} \le 80$ with simplex facet and/or simple vertex, see Table~\ref{table:polytopalpairs}.
The second column in the table explains how the polytope is found.
Polytopes $P_i$ are polytopes with $7$ or $8$ vertices known from
the classification of all polytopes with up to $8$ vertices.
Facet lists of all polytopes $P_i$ can be found in the appendix.  
\emph{Dual of} $(f_0,f_{03})$ means that the polytope is the dual of the polytope 
with polytopal pair $(f_0,f_{03})$  in the table.  
A polytope $P^*$ denotes the dual of a polytope $P$.
The methods \emph{stacking on a square pyramidal facet} and 
\emph{splitting a bipyramidal facet} 
and the polytopes $R_i(n)$ are explained above.

\begin{table}
\centering
\begin{tabular}{ll}
\begin{tabular}[t]{c|c}
 $(f_0,f_{03})$ & Description \\ \hline\hline
 \multicolumn{2}{c}{ Polytopes with $\Delta _3$-facet and simple vertex} \\\hline
 \hline
$(5,20)$& $4$-simplex  \\ \hline
$(6,26)$&  $2$-fold pyramid over quadrangle \\ \hline
$(6,29)$&   pyramid over triangular 
bipyramid   \\ \hline
$(7,29)$&   pyramid over triangular prism \\ \hline
$(7,32)$&   $2$-fold pyramid over pentagon  \\ \hline
$(7,35)$&   $P_1$ \\ \hline
$(7,36)$&   $P_2$ \\ \hline
$(7,39)$&   $P_3$ \\ \hline
$(7,45)$& $P_5$ \\ \hline
$(8,35)$&  $P_1^*$ \\ \hline
$(8,36)$& $P_2^*$ \\ \hline
$(8,38)$&  $2$-fold pyramid over hexagon \\ \hline
$(8,39)$&  $P_8$ \\ \hline
$(8,42)$&  $P_9$ \\ \hline
$(8,45)$&  $P_{11}$ \\ \hline
$(8,46)$&  $P_{12}$ \\ \hline
$(8,49)$&  $P_{13}$ \\ \hline
$(8,52)$&  $P_{14}$ \\ \hline
$(8,55)$&  $P_{15}$ \\ \hline
$(8,59)$&  $P_{16}$ \\ \hline
$(8,62)$&  $P_{18}$ \\ \hline
$(9,39)$&  $P_3^*$ \\ \hline
$(9,42)$&  $P_9^*$ \\ \hline
$(9,45)$&  split bipyramid in $(9,42)$ \\ \hline
$(9,46)$&  split bipyramid in $(9,43)$ \\ \hline
$(9,49)$&  split bipyramid in $(9,46)$ \\ \hline
$(9,52)$&  stack onto \\
&square pyramid in $(8,36)$ \\ \hline
$(10,45)$&  $P_{11}^*$ \\ \hline
$(10,46)$&  $P_{12}^*$ \\ \hline
$(10,49)$&  dual of $(9,49)$ \\ \hline
$(10,52)$&  split bipyramid in $(10,49)$ \\ \hline
$(10,55)$&   stack onto \\
&square pyramid in $(9,39)$ \\ 
\end{tabular}
&
\begin{tabular}[t]{c|c}
 $(f_0,f_{03})$ & Description \\ \hline
  \hline
$(11,45)$&  $P_5^*$ \\ \hline
$(11,49)$&  $P_{13}^*$ \\ \hline
$(11,52)$&  dual of $(9,52)$ \\ \hline
$(11,55)$&  dual of $(10,55)$ \\ \hline
$(12,52)$&  $P_{14}^*$ \\ \hline
$(13,55)$&  $P_{15}^*$ \\ \hline
   \multicolumn{2}{c}{ Polytopes with $\Delta _3$-facet} \\\hline
   \hline
$(6,36)$& cyclic polytope $C_4(6)$  \\ \hline
$(7,42)$&  $P_4$ \\ \hline
$(7,46)$&  $P_6$ \\ \hline
$(7,49)$&  $P_7$ \\ \hline
$(7,52)$&  $R_2(6)$ \\ \hline
$(7,56)$&  cyclic polytope $C_4(7)$  \\ \hline
$(8,43)$&  $P_{10}$ \\ \hline
$(8,60)$&  $P_{17}$ \\ \hline
$(8,63)$&  $P_{19}$ \\ \hline
$(8,65)$&  $P_{20}$ \\ \hline
$(8,66)$&  $P_{21}$ \\ \hline
$(8,68)$&  $P_{22}$ \\ \hline
$(8,69)$&  $P_{23}$ \\ \hline
$(8,70)$&  $P_{24}$ \\ \hline
$(8,72)$&  $P_{25}$ \\ \hline
$(8,73)$&  $P_{26}$ \\ \hline
$(8,76)$&  $P_{27}$ \\ \hline
$(8,80)$&  cyclic polytope $C_4(8)$  \\ \hline
$(9,79)$&  stack onto \\
&square pyramid in $(8,63)$\\\hline
 \hline
   \multicolumn{2}{c}{ Polytopes with simple vertex} \\
  \hline\hline
$(9,36)$& dual of cyclic polytope $C_4(6)$  \\ \hline
$(9,43)$& $P_{10}^*$  \\ \hline
$(10,42)$& $P_4^*$ \\ \hline
$(11,46)$& $P_6^*$  \\ \hline
$(12,49)$& $P_7^*$ \\ \hline
$(13,52)$& $R_2(6)^*$  \\ 
\end{tabular}
\end{tabular}
\caption{Some polytopal pairs}
\label{table:polytopalpairs}
\end{table}

Together with the inductive stacking and truncating methods from 
Section~\ref{sec:stacking}, this gives us all possible pairs for $f_{03} \le 80$
and, in particular, polytopal pairs $(f_0, f_{03})$ with simple vertex and simplicial facet, 
for $f_0 \ge 9$, $53 \le f_{03} \le 64$ and 
$4f_0 \le f_{03}$. 
See Figure~\ref{plot:f0f03_constr}.
Stacking on simplex facets and truncating 
simple vertices of these $87$ pairs of polytopes inductively
will give all polytopal pairs $(f_0,f_{03})$ 
bounded by the lower bound  $4f_0 \le f_{03}$, $f_{03} \ge 53$,
and a line with slope $12$ going through $(9,64)$.
We have hence proved the following.
\begin{lemma}\label{lem:lower_set}
There exists a $4$-polytope $P$  with $f_0(P)=f_0$ and  $f_{03}(P)=f_{03}$ whenever
\begin{align*}
4f_0 \le f_{03} \le 12f_0 -44 \text{ and } f_{03} \ge 53.
\end{align*}
\end{lemma} 
 
\begin{figure}[htbp]
  \centering
  \includegraphics[width=13.5cm]{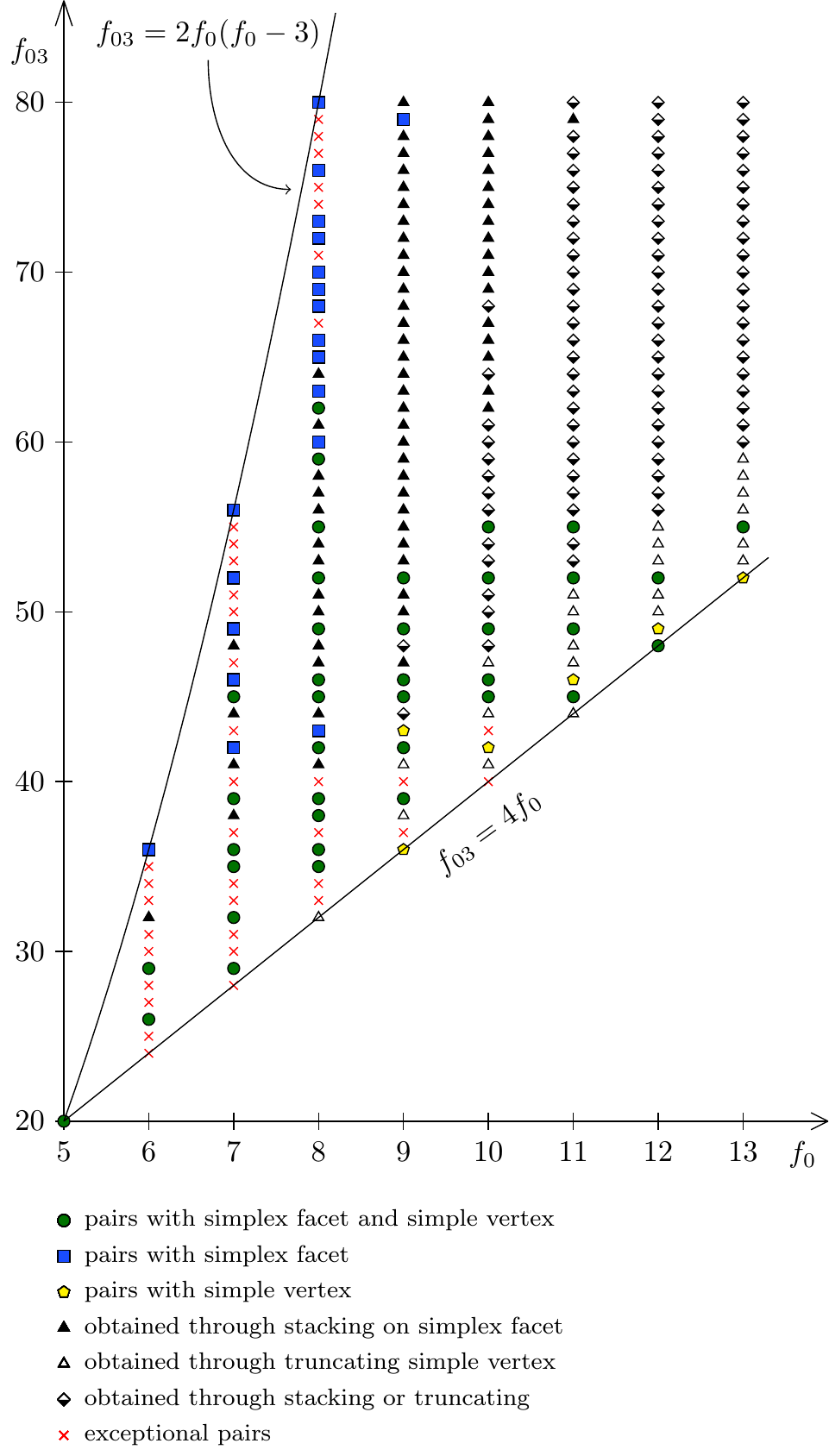}
  \caption{Polytopal pairs with $f_{03} \le 80$}
  \label{plot:f0f03_constr}
\end{figure}
In the next step we construct polytopes with
$12f_0 -44 \le f_{03} \le 2f_0(f_0 -3)$. 
In order to do so, we give examples of polytopes with simplex facet close to the upper bound.
The cyclic polytopes have polytopal pairs 
\begin{align*}
(f_0(C_4(n)),f_{03}(C_4(n)))&=(n,2n(n-3)),\\
(f_0(C_4(n+1)),f_{03}(C_4(n+1)))&=(n+1,2n(n-3)+4n-4).
\end{align*}
Our goal is to find polytopes with tetrahedral facets and polytopal pair
\[
	(n+1,2n(n-3)+i), \text{ for } i=0, \dots , 4n-5.
\]
If we find such polytopes, combined with the stacking and truncating operations from Section~\ref{sec:stacking} 
this gives us all remaining polytopal pairs.
In fact, by Lemma~\ref{thm:exceptions}, there are no polytopes with 
polytopal pair $(n+1,2n(n-3)+4n-k)$ for $k \in \{5,6,7,9,10,13,17\}$.
In these cases, the ``next'' polytope in the stacking process, 
a polytope with polytopal pair $(n+2,2n(n-3)+4n-k+12$) for $k \in \{ 5,6,7,9,10,13\}$,
is given for $m=n+1$ by the polytope with $(m+1,2m(m-3) +16-k)$.
For $k=17$, the polytope with polytopal pair $(m+1,2m(m-3) -1)$
can be obtained through stacking a vertex onto two facets of $\delta^* _3(n-3,n)$ 
with polytopal pair $(n+1,2n(n-3)+4n-21)=(m,2m(m-3)-17)$
(see Section~\ref{sec:dualcyclic}).
Stacking a vertex onto $\delta^* _3(n-3,n)$, 
such that the vertex is beyond two simplex facets which have a common $2$-face, yields a new 
polytope with $16$ more vertex-facet incidences and one additional
vertex (cf.\ Section~\ref{sec:genstacking}).
So the new polytope has the required polytopal pair $(m+1,2m(m-3) -1)$.
 
To find examples of polytopes with $(f_0,f_{03})=(n+1,2n(n-3)+i)$, for $f_0 = n+1 \ge 9$,
$i=0, \dots , 4n-4$, $i \ne 4n -j$ for $j \in \{ 5,6,7,9,10,13,17\}$,
we  use the constructions from Sections~\ref{sec:genstacking} and \ref{sec:facetsplitting}.
Table~\ref{table:(n+1,2n(n-3)+i)} shows how the polytopes are constructed.
\begin{table}
\centering
\begin{tabular}{l|l|l} 
\multirow{10}{*}{ $f_{03} \equiv 0$ mod 4}
& $f_{03}$ & Example of polytope\\ \hline
&$2n(n-3)$& stack onto $\Delta _3$-facet of $R_{n-7}(n-1)$ \\ \cline{2-3}
&$2n(n-3)+4$& stack onto $\Delta _3$-facet of $R_{n-6}(n-1)$ \\ \cline{2-3}
&$2n(n-3)+8$& stack onto $\Delta _3$-facet of $R_{n-5}(n-1)$ \\ \cline{2-3}
&$2n(n-3)+12$& stack onto $\Delta _3$-facet of $C_4(n)$\\ \cline{2-3}
&$2n(n-3)+16$& $R_{2}(n)$ \\ \cline{2-3}
&$2n(n-3)+20$& $R_{3}(n)$ \\ \cline{2-3}
&\dots & \dots \\ \cline{2-3}
&$2n(n-3)+4n-8$& $R_{n-4}(n)$  \\ \cline{2-3}
&$2n(n-3)+4n-4$& $C_4(n+1)$ \\ \hline
 \hline
\multirow{6}{*}{ $f_{03} \equiv 1$ mod 4}
&$2n(n-3)+1$ &  stack onto $\Delta _3$-facet of $\delta^* _1(n-4,n-1)$ \\ \cline{2-3}
&$2n(n-3)+5$ & stack onto $\Delta _3$-facet of $\delta^* _1(n-3,n-1)$ \\ \cline{2-3}
&$2n(n-3)+9$ &  $\delta^* _1(3,n)$\\ \cline{2-3}
&\dots & \dots \\ \cline{2-3}
&$2n(n-3)+4n-11$ &  $\delta^* _1(n-2,n)$ \\ \cline{2-3} 
&$2n(n-3)+4n-7$& does not exist\\ \hline \hline
\multirow{6}{*}{ $f_{03} \equiv 2$ mod 4}
&$2n(n-3)+2$& stack onto $\Delta _3$-facet of $\delta^* _2(n-3,n-1)$ \\ \cline{2-3}
&$2n(n-3)+6$& $\delta^* _2(3,n)$ \\ \cline{2-3}
&\dots & \dots \\ \cline{2-3}
&$2n(n-3)+4n-14$& $\delta^* _2(n-2,n)$ \\ \cline{2-3}
&$2n(n-3)+4n-10$& does not exist\\ \cline{2-3}
&$2n(n-3)+4n-6$& does not exist \\ \hline \hline
\multirow{7}{*}{ $f_{03} \equiv 3$ mod 4}
&$2n(n-3)+3$&  $\delta^* _3(3,n)$\\ \cline{2-3}
&\dots & \dots \\ \cline{2-3}
&$2n(n-3)+4n-21$& $\delta^* _3(n-3,n)$\\ \cline{2-3}
&$2n(n-3)+4n-17$& does not exist   \\ \cline{2-3}
&$2n(n-3)+4n-13$& does not exist   \\ \cline{2-3}
&$2n(n-3)+4n-9$& does not exist  \\ \cline{2-3}
&$2n(n-3)+4n-5$& does not exist \\  
\end{tabular}
\caption{Polytopal pairs $(n+1,2n(n-3)+i)$, $n \ge 8$}
\label{table:(n+1,2n(n-3)+i)}
\end{table}
 
For $f_0 \le 8$  we use the fact that polytopes with up to $8$ vertices have been classified
(see Table~\ref{table:polytopalpairs}).
In particular, we can construct examples of polytopes with simplex facet, 
simple vertex and polytopal pair
\[
	(n+2,2n(n-3)+i), \text{ for all } i=0, \dots , 4n-5, \ n \ge 7.
\]
If we now inductively stack vertices on simplex facets and truncate simple vertices,
we obtain polytopes with polytopal pairs $(f_0,f_{03})$ with $f_0 \ge 9$
bounded from above by $2f_0(f_0-3)$ and from below by a line of slope 4,
going through $(9,56)$.
So we have found all polytopal pairs with
\begin{align}\label{ineq:upper_set} 
4f_0 +20 \le f_{03} \le 2f_0(f_0-3)
\end{align}
for all $f_0 \ge 9$, with the only exceptions for each value of $f_0$ being the $7$ 
pairs mentioned above.
Lemma~\ref{lem:lower_set} and   Inequality~(\ref{ineq:upper_set}) together give all pairs 
$(f_0,f_{03})$ with $f_0 \ge 9$, $f_{03} \ge 53$ 
within the bounds, excluding the exceptional pairs.
Since we classified all possible polytopal pairs with $f_{03} \le 80$,
and  in particular all polytopal pairs with $f_0 \le 8$, we have now proved Theorem~\ref{main_thm1}.

\subsection{Other flag vector pairs}\label{sec:otherflagpairs}
The flag vector of a $4$-polytope has $16$ entries.
Besides $f_{\emptyset} = 1$, the following nine entries depend on only one other entry:
\begin{alignat*}{6}
f_{01} &= 2f_1,&&f_{12}&&= f_{02}, &&f_{13} &&= f_{02},\\
f_{23} &= 2f_2, &&f_{012} &&= 2f_{02}, \ \ &&f_{013} &&= 2f_{02},\\
f_{023} &= 2f_{02}, \ \ &&f_{123} &&= 2f_{02}, &&f_{0123} &&= 4f_{02}.
\end{alignat*}
These equations are some of the  Generalized Dehn--Sommerville equations for $4$-dimensional polytopes 
(Lemma~\ref{genDehnSommerville}). To obtain all $2$-dimensional coordinate projections of the flag vectors 
of $4$-polytopes, we therefore only have to consider the six entries
$f_0$, $f_1$, $f_2$, $f_3$, $f_{02}$ and $f_{03}$.
We still need to determine the projections 
\[
	\PiFset{4}{0,02}{}, \PiFset{4}{1,02}{}, \PiFset{4}{1,03}{} \text{ and }\PiFset{4}{02,03}{}.
\]
All other cases have already been done, or they follow directly, 
either by duality or by the linear dependence on a single entry.

For the projections $\PiFset{4}{0,02}{}$ and $\PiFset{4}{1,02}{}$, 
the pairs $(f_0,f_{02})$ in $\PiFset{4}{0,02}{}$  satisfy the fairly obvious bounds 
$6f_0\le f_{02}\le 3f_0(f_0-3)$. 
Equality holds for simple and neighborly polytopes, respectively.
Similarly the pairs $(f_1,f_{02})$ in $\PiFset{4}{1,02}{}$
satisfy
$3f_1\le f_{02} \le 6f_1-3\sqrt{8f_1+1}-3$,
with equality for 2-simple polytopes (each edge is contained in exactly 3 facets)
and neighborly polytopes, respectively.

The projection sets  $\PiFset{4}{1,03}{} $ and  $\PiFset{4}{02,03}{}$ are more difficult to describe even approximately. 
Upper bounds for $f_{03}$ in terms of $f_1$ are achieved for neighborly polytopes,
and in terms of $f_{02}$ for center boolean polytopes.
The problem of finding tight lower bounds for $f_{03}$  in terms of $f_1$ and $f_{02}$ is related to 
the open problem of finding an upper bound for the \emph{fatness}
$F=\frac{f_1+f_2-20}{f_0+f_3-10}$
of a polytope~\cite{EKZ03}.%

\section{Face vector pair $(f_0,f_{d-1})$ for $d$-polytopes}\label{sec:f0fd-1}

Now we work towards analogous results in higher dimensions. In one instance, recently
the projection $\PiFset{5}{0,1}{}$ of the $f$-vector of $5$-polytopes to $(f_0,f_1)$ was determined almost simultaneously by Kusunoki and Murai~\cite{KusunokiMurai} and by Pineda-Villavicencio, Ugon and 
Yost~\cite{PVUY17}.

We consider $\PiFset{d}{0,d-1}{}$, the projection of the set of $f$-vectors of $d$-polytopes to $(f_0,f_{d-1})$.
 In the following, for given $d$, we will consider pairs of integers $(n,m)$
and analyze under which conditions there are $d$-polytopes with $n$ vertices and $m$ facets.

\begin{definition}
For fixed dimension $d$, a pair $(n,m)\in\N^n$ is \emph{$d$-large} if $n+m\ge \binom{3d+1}{\floor{d/2}}$; 
it is \emph{$d$-small} otherwise. A pair $(n,m)$ will be called an \emph{exceptional pair} if
$m\le f_{d-1}(C_d(n))$ and $n\le f_{d-1}(C_d(m))$, and if there is no $d$-polytope with $n$ vertices and $m$ facets.
\end{definition}

The situation looks as follows:
\begin{enumerate}
\item[(1)]
If $P$ is a $d$-polytope with $n$ vertices and $m$ facets,
then 
\[
	m\le f_{d-1}(C_d(n)), \ n\le f_{d-1}(C_d(m)).
\]
\item[(2)]
If $(n,m)$ is a pair of integers, $n,m \ge d+1$ such that 
for a given  dimension $d$,\\
$m\le f_{d-1}(C_d(n)), \ n\le f_{d-1}(C_d(m))$,
then there \emph{usually} exists a $d$-polytope with $n$ vertices and $m$ facets:
\begin{itemize}
\item[(2.1)]For $d\le 4$ no exceptional pairs exist.

\item[(2.2)] For even $d\ge 6$, only finitely many exceptional pairs exist, all of which are $d$-small
(see Figure~\ref{plot:f0,fd-1}a).

\item[(2.3)] For odd $d\ge 5$, there exist finitely many $d$-small exceptional pairs and additionally 
infinitely many $d$-large exceptional pairs
for $m$ odd and  $f_{d-1}(C_d(n-1))<m<f_{d-1}(C_d(n))$
and for $n$ odd and $f_{d-1}(C_d(m-1))<n<f_{d-1}(C_d(m))$
(see Figure~\ref{plot:f0,fd-1}b).
\end{itemize}
\end{enumerate}
(1) are the UBT inequalities.

\noindent
(2.1) holds trivially for $d\le 2$. 
For dimension $3$, it is given by Steinitz' classification of all $3$-dimensional polytopes \cite{Steinitz3}. 
For dimension $4$, 
this is Theorem~\ref{thm:f0f3}.

\noindent
(2.2) is Theorem~\ref{thm:f0fd-1_evendim} and (2.3) is Theorem~\ref{thm:f0fd-1_odddim}.

\begin{figure}
  \centering
  \includegraphics[width=.8\linewidth]{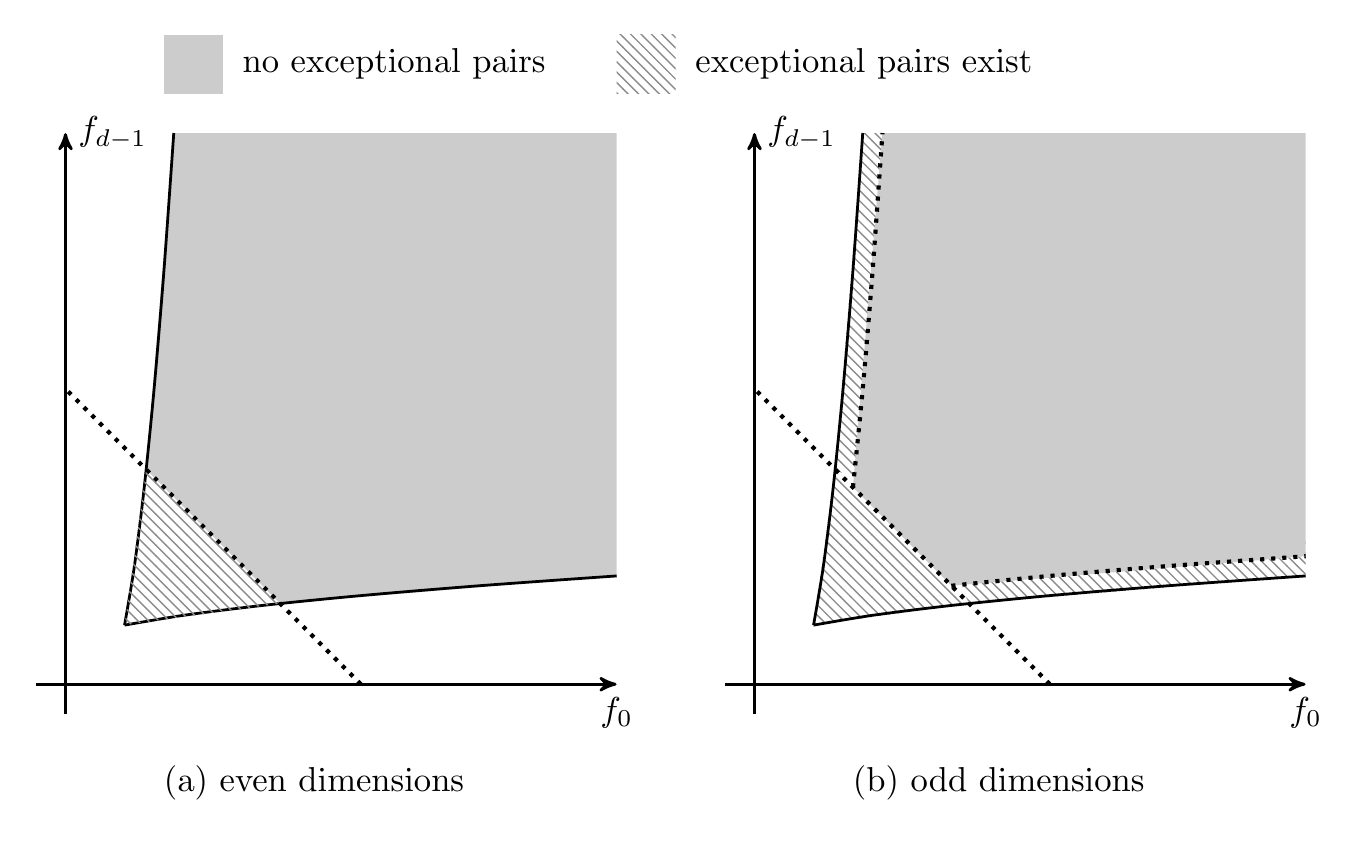} 
\caption{Projections $\PiFset{d}{0,d-1}{}$}
\label{plot:f0,fd-1}
\end{figure}

\begin{theorem}\label{thm:f0fd-1_evendim}
Let $d\ge 2$ be even and $(n,m)$ $d$-large.
Then there exists a $d$-polytope $P$ with $n$ vertices and $m$ facets if and only if
\[
	m\le f_{d-1}(C_d(n)) \text{ and } n\le f_{d-1}(C_d(m)).
\]
The first inequality holds with equality if and only if $P$ is neighborly,
and the second inequality holds with equality if and only if $P$ is dual-neighborly.

However, for $d\ge 6$ \ $d$-small exceptional pairs $(n,m)$ exist. 
\end{theorem}
\begin{proof}
The necessity of the conditions and the equality cases are  direct consequences of the upper bound theorem 
(McMullen~\cite{McMullen}).
For the sufficiency,
consider the $g$-vector of simplicial polytopes.

The $\frac d2$th entry of the $g$-vector of
 a cyclic polytope $C_{d}(n)$ in even dimension $d=2k$ with $n$ vertices is
\[
	g_{d/2}(C_{d}(n))=g_k(C_{2k}(n))=\binom{n-k-2}{k}.
\]
A consequence of the sufficiency part of the $g$-theorem (Billera \& Lee~\cite{BilleraLee1,BilleraLee2}) is that
there exist simplicial $2k$-polytopes with $n$ vertices, $g_i =g_i(C_{2k}(n))$ and $g_k=l$
for all $1\le i \le k-1$ and for all $0 \le l \le
 \binom{n-k-2}{k}$.
 \newline
For all simplicial $2k$-polytopes, 
\[
	f_{2k-1}=(2k+1)+g_1(2k-1)+g_2(2k-3)+\dots +g_{k-1}\cdot 3+g_k.
\]
Hence, there exist simplicial $2k$-polytopes with $n$ vertices and $f_{2k-1}(C_{2k}(n))-l$ facets,
for $0 \le l \le \binom{n-k-2}{k}$.
Observe that $\binom{n-k-2}{k}> f_{2k-1}(C_{2k}(n))-f_{2k-1}(C_{2k}(n-1))$ for large $n$.
In particular, this inequality holds for 
$n\ge 7k+2 =\frac 7 2 d +2$.
This means that for $n \ge \frac 7 2 d +2$ there are simplicial $2k$-polytopes with $n$ vertices and 
$m$ facets, for all integers $m$ such that $f_{2k-1}(C_{2k}(n-1))\le m \le f_{2k-1}(C_{2k}(n))$.
Now we can stack a vertex on a facet of each of these polytopes and obtain polytopes
with one more vertex and $d-1$ more facets. The new polytope has a simple vertex and simplex facets,
so we can repeatedly stack vertices on simplex facets and truncate simple vertices.
Truncating simple vertices gives a polytope with $d-1$ more (simple) vertices and one additional (simplex) facet.
Consider the pair $(7k+2,m)$: We have just seen that this pair is not an exceptional pair as long as $m\ge f_{2k-1}(C_{2k}(7k+1))$.
Stacking a vertex on a facet of a polytope with pair $(7k+2, f_{2k-1}(C_{2k}(7k+1)))$ gives a polytope with 
simplex facet, simple vertex and pair $(n_0, m_0):=(7k+3, f_{2k-1}(C_{2k}(7k+1))+d-1)$. 
\newpage
\noindent
Consider the line $\ell_1$ of slope $\frac{1}{d-1}$ through $(n_0, m_0)$.
There are no exceptional pairs with $n\ge n_0$ above $\ell_1$.
The line $\ell_1$ intersects the line $\ell_2:m=n$ in a pair $(n,n)$ such that 
\[
	n=\frac{k+1}{12k+2}\binom{6k+1}{k}< \frac{1}{2}\binom{6k+1}{k}.
\] 
Together with the dual polytope, we have obtained all polytopes with pairs $(n,m)$
within the bounds such that  
\[
	\binom{3d+1}{\frac d 2}\le n+m.
\]
Hence, there are no $d$-large exceptional pairs. 
 
On the other hand, there are exceptional pairs for $d$-small $(n,m)$.
As an example consider $d$-polytopes with $d+2$ vertices.
All $d$-polytopes with $d+2$ vertices are simplicial or (multiple) pyramids over some $r$-polytope with $r+2$ vertices \cite[Sect.~6.5]{Ziegler}.

There are exactly $\floor{\frac{d}{2}}=k$ different combinatorial types of simplicial $d$-polytopes with $d+2$ vertices
(\cite[Sect.~6.1]{Grunbaum}).
One of these types is the stacked polytope with $2d$ facets. In particular, for $d\ge 6$, $2d \le k^2+k+1$.
Any non-simplicial $d$-polytope with $d+2$ vertices is a pyramid and has thus at most
$f_{d-1}(\text{Pyr}(C_{d-1}(d+1)))=f_{d-2}(C_{d-1}(d+1))+1=k^2+k+1$ facets for $d=2k$. 
This means that there are at most $k-1$ different combinatorial types of $(2k)$-polytopes with $2k+2$ vertices
and more than $k^2+k+1$ facets.

The cyclic polytope with $d+2$ vertices has $k^2+2k+1$ facets for even dimensions $d=2k$. 
So there are $k$ pairs $(n,m)$ for given $n$ and 
$k^2+k+1 < m\le f_{d-1}(C_d(n))$, but at most $k-1$ combinatorially non-equivalent polytopes.
Therefore, for  $n=d+2$ and even $d\ge 6$ there must be at least one exceptional pair.
\qed
\end{proof}

An example is the pair $(n,m)=(8,14)$ for dimension $6$: There is no 
$6$-polytope with $8$ vertices and $14$ facets \cite{FMMclass},
but there are $6$-polytopes with $8$ vertices and $13$ or $15$ facets.

\begin{theorem}\label{thm:f0fd-1_odddim}
Let $d\ge 3$ be odd.
If $(n,m)$ is $d$-large, then there exist $d$-polytopes with $n$ vertices and $m$ facets if and only if 
\[
	m\le f_{d-1}(C_d(n)) \text{ and } n\le f_{d-1}(C_d(m))
\]
with $d$-large exceptional pairs occurring only for $d\ge 5$, if $m$ is odd and $f_{d-1}(C_d(n-1)\le m$
and if $n$ is odd and $f_{d-1}(C_d(m-1)\le n$.

However, for $d\ge 5$ \ $d$-small exceptional pairs $(n,m)$ exist.
\end{theorem} 

\begin{proof}
The  necessity follows again from the upper bound theorem (McMullen~\cite{McMullen}).
For the sufficiency, 
we follow the proof of Theorem~\ref{thm:f0fd-1_evendim}.

A cyclic polytope $C_{d}(n)$ in odd dimension $d=2k+1$ with $n$ vertices has $g_{\floor{\frac{d}{2}}}$ equal to
\[
	g_{\floor{\frac{d}{2}}}(C_{d}(n))=g_k(C_{2k+1}(n))=\binom{n-k-3}{k}.
\]
Again, by the $g$-theorem~\cite{BilleraLee1,BilleraLee2,Stanley},
there exist simplicial $(2k+1)$-polytopes with $n$ vertices, $g_i =g_i(C_{2k}(n))$ and $g_k=l$
for all $1\le i \le k-1$ and for all $0 \le l \le \binom{n-k-3}{k}$.

For all simplicial $(2k+1)$-polytopes, 
\[
	f_{d-1}=(d+1)+g_1(d-1)+g_2(d-3)+\dots +g_{k-1}\cdot 4+g_k \cdot 2.
\]
Hence, there exist simplicial $(2k+1)$-polytopes with $n$ vertices and $f_{2k}(C_{2k+1}(n))-2l$ facets,
for $0 \le l \le \binom{n-k-3}{k}$.

We have that $2\binom{n-k-3}{k}> f_{2k}(C_{2k+1}(n))-f_{2k}(C_{2k+1}(n-1))$ holds for large $n$,
in particular for $d=5$ if $n\ge 9$ and for general $d$ if $n\ge 5k+1=\frac 5 2 d-\frac 3 2$. 
With the same calculations as before, we obtain polytopes with $n$ vertices and $m$ facets
for all pairs $(n,m)$ if $n$ and $m$ are even and if
\[
	n+m \ge  \frac{2}{2k-1}(4k\binom{4k-1}{k}+4k^2-5k-2).
\] 
For $d\ge 7$, this implies that
\[
	n+m\ge \binom{6k+4}{k}=\binom{3d+1}{\left\lfloor\frac{d}{2}\right\rfloor}.
\]
For $d=5$, we check that the constructions give us all polytopes with $n+m\ge 58$,
where $\binom{3\cdot 5+1}{2}>58$.
We can also construct polytopes with an odd number of facets, as long as $m\le f_{d-1}(C_d(n-1))$.
For this, we need a generalized stacking construction similar to the one described in  
Section~\ref{sec:genstacking}.
Starting with a simplicial polytope, we place a new vertex beyond one facet, inside the affine hull of 
a second facet and beneath all other facets.
The new polytope has one new (simple) vertex and $d-2$ new facets. The polytope has one facet which is a bipyramid
over a triangle. All other facets are simplices, so we can apply the inductive stacking and truncating method from before.

There are exceptional pairs $(n,m)$ if $m$ is odd and close to $f_{d-1}(C_d(n))$:
non-simplicial $d$-polytopes with $n$ vertices have at most $f_{d-1}(C_{d}(n))-\floor{\frac{d}{2}}$ facets.
This is a direct consequence of the upper bound theorem for \emph{almost simplicial polytopes}
by Nevo, Pineda-Villavicencio, Ugon \& Yost~\cite{NPVUY17}. 
These authors give upper bounds for the number of faces of  the family $\mathcal{P}(d,n,s)$ of \emph{almost simplicial polytopes},  $d$-polytopes on $n$ vertices where
one facet has $d+s\ge d+1$ vertices and all other facets  are simplices.
Such polytopes have at most 
$f_{d-1}(C_{d}(n))-\floor{\frac{d}{2}}$ facets.
(This follows  from \cite{NPVUY17}, Thm.~1.2 and Prop.~4.2.)

For any non-simplicial polytope $P$ on $n$ vertices there exists an almost-simplicial polytope on $n$ vertices
(i.e. a polytope with exactly one non-simplicial facet) that has at least as many $i$-faces as $P$:
Let $F$ be a non-simplicial facet of $P$. If we successively pull every vertex of $\ver P\backslash \ver F$ 
(in the sense of \cite{EGK}) 
and then pull every vertex $v\in \ver F$ \emph{within} the affine hull of $F$, then the resulting polytope is almost simplicial, with at least as many $i$-dimensional faces as $P$.
So the $i$-faces of non-simplicial $d$-polytopes on $n$ vertices are maximized
among the almost simplicial $d$-polytopes on $n$ vertices.
In particular, for any non-simplicial $d$-polytope $P$, 
$f_{d-1}(P)\le f_{d-1}(C_{d}(n))-\floor{\frac{d}{2}}$.
Thus, for odd $d$, odd $m$, and 
\[
	f_{d-1}(C_{d}(n))-\left\lfloor\tfrac{d}{2}\right\rfloor<m<f_{d-1}(C_{d}(n)),
\]
$(n,m)$ is an exceptional pair.

The rest of the theorem for $d$-large $(n,m)$ follows by duality. 
For $d$-small $(n,m)$, the non-constructive proof for the existence of exceptional pairs 
in the even-dimensional case works as well in the odd-dimensional case. 
It can be slightly improved:
All $d$-polytopes with $d+2$ vertices are simplicial or (multiple) pyramids over some $r$-polytope with $r+2$ vertices \cite[Sect.~6.5]{Ziegler}.
In particular, for $d=2k+1$ and odd $m$, any polytope $P$ with $d+2$ vertices and $m$ facets is a pyramid over some $(d-1)$-polytope $Q$ with $d+1$ vertices and $m-1$ facets. Hence,
\begin{align*}
m-1 = f_{d-2}(Q) \le f_{d-2}(C_{d-1}(d+1))=k^2+2k+1.
\end{align*}
Comparing this to 
\[
	f_{d-1}(C_{d}(d+2))=k^2+3k+2,
\]
we see that there are $\left\lfloor\frac{k}{2}\right\rfloor$ exceptional pairs $(n,m)$ for which there are no 
$(2k+1)$-polytopes, such that $m$ is odd and
\[
	k^2+2k+2<m<k^2+3k+2.
\]
\qed
\end{proof}

\paragraph{Remark.} 
This implies that for for odd $d$ the projection sets  $\PiFset{d}{0,d-1}{}$ have infinitely many exceptional pairs, all
of them near the boundary.
For a complete characterization of $d$-large pairs in $\PiFset{d}{0,d-1}{}$
one would need to analyze closely the possible facet numbers of non-simplicial polytopes with many facets.

For low dimensions, we can improve the bounds for the $d$-\emph{large} pairs. 
We have seen that in dimension $5$, a pair can be called $d$-large if $n+m\ge 58$.
Similarly, for dimension $6$, the bound for $d$-large pairs can be reduced to $n+m\ge 132$: 
It can be seen from the $g$-theorem that simplicial $6$-polytopes with $n$ vertices have
$5n-28$, $5n-25$, $5n-24$, or $5n-22$ to $f_5(C_6(n))$ facets.
For $n\ge 11$, it holds that $5n-22<f_5(C_6(n-1))$. 
From this, the bound $n+m\ge 132$ for $d$-large pairs can be derived.

\newpage

\section*{Appendix}\label{appendix}

Table~\ref{table:7to8vertices} lists all polytopes $P_i$ with $7$ and $8$ vertices from 
Table~\ref{table:polytopalpairs} used in the construction of all possible pairs $(f_0,f_{03})$.
The polytopes are given by their facet list. See Fukuda, Miyata \& Moriyama~\cite{FMMclass}
for a  complete list of 
all $31$ polytopes with $7$ vertices and all $1294$ polytopes with $8$ vertices.
Entry~$7.x$ in the last column means that the polytope can be found as the $x$th polytope listed 
in the classification of $4$-polytopes with $7$ vertices.
\begin{table}[h]
\begin{tabular}{l|l|l} 
polytope & facet list & row  \\ \hline
$P_1$ & [654321][65430][6520][6420][5310][5210][4310][4210]&  7.3 \\  \hline
$P_2$ & [65432][65431][65210][64210][5320][5310][4320][4310] &  7.21\\ \hline
$P_3$ &[65432][65431][65210][6421][5320][5310][4320][4310][4210] &  7.22\\ \hline
$P_4$ &[65432][65410][6531][6431][5420][5321][5210][4320][4310][3210] &  7.11\\ \hline
$P_5$ & [65432][6541][6531][6431][5421][5320][5310][5210][4320][4310][4210] &  7.16\\ \hline
$P_6$ & [65432][65431][6521][6420][6410][6210][5320][5310][5210][4320][4310]&  7.24\\ \hline
$P_7$ & [65432][6541][6531][6430][6410][6310][5421][5320][5310][5210][4320][4210]&  7.13\\ \hline
$P_8$ & [765432][765410][76321][75310][64210][5430][4320][3210] &  8.186\\ \hline
$P_9$ &[765432][76541][76310][75310][64210][6320][5420][5410][5320] &  8.285\\ \hline
$P_{10}$ &[76543][76542][76321][75310][75210][64310][64210][5430][5420] &  8.1145\\ \hline
$P_{11}$ & [765432][76541][76310][54310][7531][6421][6320][6210][4320][4210]&  8.241\\ \hline
$P_{12}$ & [765432][76541][76320][75310][54310][7610][6421][6210][4320][4210]&  8.353\\ \hline
$P_{13}$ &[765432][76541][73210][63210][7631][7520][7510][6420][6410][5420][5410]&  8.201\\ \hline
$P_{14}$ & [765432][76541][76310][7531][6430][6410][5420][5410][5321][5210][4320][3210] &  8.306\\ 
\hline
$P_{15}$ &[765432][76510][7641][7541][6530][6421][6321][6310][5420][5410][5320][4210]&  8.117\\ 
&[3210]& \\ \hline
$P_{16}$ &[76543][76521][76420][7542][6531][6431][6410][6210][5432][5320][5310][5210]&  8.676\\ 
&[4320][4310]& \\ \hline
$P_{17}$ &[76543][76542][73210][63210][7632][7531][7520][7510][6431][6420][6410][5431]&  8.909\\ 
&[5420][5410] & \\ \hline
$P_{18}$ &[76543][76521][7642][7542][6530][6510][6432][6320][6210][5430][5421][5410]&  8.778\\ 
&[4321][4310][3210]& \\ \hline
$P_{19}$ & [76543][76542][73210][7631][7621][7530][7520][6431][6420][6410][6210][5431]&  8.910\\ 
&[5420][5410][5310] & \\ \hline
$P_{20}$ &[76543][7652][7642][7531][7521][7431][7421][6530][6521][6510][6430][6420]&  8.805\\ 
&[6210][5310][4310][4210]& \\ \hline 
$P_{21}$ &[76543][76542][7632][7531][7521][7320][7310][7210][6431][6420][6410][6320]&  8.1227\\
&[6310][5431][5421][4210] & \\ \hline
$P_{22}$ &[7654][7653][7643][7542][7532][7431][7421][7321][6540][6530][6431][6410]&  8.1262\\ 
&[6310][5420][5320][4210][3210] & \\ \hline
$P_{23}$ &[76543][7652][7642][7531][7521][7431][7421][6530][6521][6510][6430][6420]&  8.806\\ 
&[6210][5310][4321][4320][3210] & \\ \hline
$P_{24}$ &[76543][76542][7631][7621][7531][7520][7510][7210][6430][6420][6321][6320]&  8.1041\\ 
&[5431][5420][5410][4310][3210] & \\ \hline
$P_{25}$ &[7654][7653][7643][7542][7532][7431][7421][7321][6542][6530][6520][6430][6420]&  8.1263\\ 
&[5321][5310][5210][4310][4210]& \\ \hline
$P_{26}$ &[76543][7652][7642][7541][7521][7420][7410][7210][6530][6521][6510][6432]&  8.815\\ 
&[6320][6210][5431][5310][4320][4310] & \\ \hline
$P_{27}$ &[7654][7653][7643][7542][7532][7431][7421][7321][6542][6530][6520][6431][6420]&  8.1266\\ 
&[6410][6310][5321][5310][5210][4210]& \\ 
\end{tabular}
\caption{Polytopes $P_i$ with 7 and 8 vertices}
\label{table:7to8vertices}
\end{table} 
\end{document}